\documentclass[a4paper, 10pt, oneside, onecolumn]{article}

\RequirePackage{geometry}
\usepackage[utf8]{inputenc}
\usepackage[T1]{fontenc}
\usepackage{lmodern}
\usepackage{amsmath}
\usepackage{amssymb}
\usepackage{amsfonts}
\usepackage{amsthm}
\usepackage{xcolor}
\usepackage{hyperref}
\usepackage{cleveref}
\usepackage{mathtools}

\hypersetup{
	colorlinks=false,
	pdfborder={0 0 0},
	pdftitle={Two new functional inequalities and their application to the eventual smoothness of solutions to a chemotaxis-Navier-Stokes system with rotational flux},
	pdfauthor={Frederic Heihoff},
	pdfkeywords={functional inequalities; variational methods; Trudinger--Moser inequality; chemotaxis; Navier--Stokes; generalized solutions; eventual smoothness},
	bookmarksopen=true,
}

\newtheorem{theoremext}{Theorem}

\newtheorem{base}{Base}[section]
\numberwithin{equation}{section}

\theoremstyle{plain}
\newtheorem{theorem}[base]{Theorem}
\newtheorem{lemma}[base]{Lemma}

\newtheorem{corollary}[base]{Corollary}

\theoremstyle{definition}
\newtheorem{definition}[base]{Definition}

\newtheorem{remark}[base]{Remark}

\newcommand{\R}{\mathbb{R}}
\newcommand{\N}{\mathbb{N}}
\renewcommand{\d}{\,\mathrm{d}}
\newcommand{\laplace}{\Delta}
\newcommand{\grad}{\nabla}
\renewcommand{\div}{\nabla \cdot}
\renewcommand{\L}[1]{{L^{#1}(\Omega)}}
\newcommand{\defs}{\coloneqq}
\newcommand{\sfed}{\eqqcolon}
\newcommand{\feds}{\eqqcolon}
\newcommand{\stext}[1]{\;\;\text{ #1 }\;\;}
\newcommand{\eps}{\varepsilon}
\newcommand{\loc}{\mathrm{loc}}

\newcommand\numberthis{\addtocounter{equation}{1}\tag{\theequation}}

\makeatletter
\g@addto@macro\bfseries{\boldmath}
\makeatother

\begin{document}
\title{Two new functional inequalities and their application to the eventual smoothness of solutions to a chemotaxis-Navier--Stokes system with rotational flux}
\author{
	Frederic Heihoff\footnote{fheihoff@math.uni-paderborn.de}\\
	{\small Institut f\"ur Mathematik, Universit\"at Paderborn,}\\
	{\small 33098 Paderborn, Germany}
}
\date{}
\maketitle

\begin{abstract}
	\noindent We prove two new functional inequalities of the forms
	\[
		\int_G \varphi (\psi - \overline{\psi}) \leq \frac{1}{a}\int_G \psi \ln \left(\frac{\;\psi\;}{ \overline{\psi}}\right) + \frac{a}{4\beta_0} \left\{ \int_G \psi \right\}\int_G|\nabla \varphi|^2
	\]
	and
	\[
		\int_G \psi \ln \left(\frac{\;\psi\;}{ \overline{\psi}}\right) \leq \frac{1}{\beta_0}\left\{ \int_G \psi \right\}\int_G |\nabla \ln(\psi)|^2
	\]
	for any finitely connected, bounded $C^2$-domain $G \subseteq \mathbb{R}^2$, a constant $\beta_0 > 0$, any $a > 0$ and sufficiently regular functions $\varphi$, $\psi$. 
	\\[0.5em]
	We then illustrate their usefulness by proving long time stabilization and eventual smoothness properties for certain generalized solutions to the chemotaxis-Navier--Stokes system
	\[
	\left\{\;\;
	\begin{aligned}
	n_t + u \cdot \nabla n &\;\;=\;\; \Delta n - \nabla \cdot (nS(x,n,c) \nabla c),\\
	c_t + u\cdot \nabla c &\;\;=\;\; \Delta c - n f(c), \\ 
	u_t + (u\cdot \nabla) u  &\;\;=\;\; \Delta u + \nabla P + n \nabla \phi, \;\;\;\;\;\; \nabla \cdot u = 0,
	\end{aligned}
	\right.	
	\]
	on a smooth, bounded, convex domain $\Omega \subseteq \mathbb{R}^2$ with no-flux boundary conditions for $n$ and $c$ as well as a Dirichlet boundary condition for $u$. We further allow for a general chemotactic sensitivity $S$ attaining values in $\mathbb{R}^{2\times 2}$ as opposed to a scalar one. 
	\\[0.5em]
\end{abstract}
\pagebreak
\section{Introduction}
\paragraph{Two new functional inequalities.} As it explores the space at the limits of the Sobolev inequalities, the Trudinger--Moser inequality (cf.\  \cite{moser1971sharp}, \cite{trudinger1967imbeddings}) and its corollaries have proven crucial in discovering the structural subtleties of many partial differential equations. Especially in the case of two-dimensional domains, where the gap between $W^{1,2}$ embedding into $L^p$, $p\in[1,\infty)$, but not embedding into $L^\infty$ seems particularly vast, many interesting properties at the parameter boundaries of the sharp Trudinger--Moser inequality have been discovered. One such example is the existence of blowing-up solutions to the mean field equation (cf.\ \cite{EspositoExistenceBlowingupSolutions2005} or \cite{AdimurthiGlobalCompactnessProperties2000} for a similar discussion in a slightly different setting), where the mentioned blowup occurs as a critical system parameter approaches a value connected to the Trudinger--Moser inequality. Similarly for the two-dimensional Keller--Segel system (cf.\ \cite{keller1970initiation}), it has been proven that blowup behavior of solutions depends critically on the initial mass of the first solution component, where the value of said critical mass is again closely connected to the optimal parameter in the sharp Trudinger--Moser inequality (cf.\ \cite{HorstmannBlowupChemotaxisModel2001}, \cite{nagai1997application}). Apart from these already striking results, the Trudinger--Moser inequality has also been used to cope with exponential nonlinearities in the wave equations (cf.\ \cite{AlvesExistenceUniformDecay2009a}) as well as the heat equation (cf.\ \cite{ArrietaParabolicProblemsNonlinear1999}) among other examples.
\\[0.5em]
One recently derived consequence of the Trudinger--Moser inequality, which is e.g.\ used in the existence theory of chemotaxis-Navier--Stokes systems, are inequalities of the forms
\[
	\int_G \varphi (\psi - \overline{\psi}) \leq \frac{1}{a}\int_G \psi \ln \left(\frac{\;\psi\;}{ \overline{\psi}}\right) + \frac{a}{4\beta_0} \left\{ \int_G \psi \right\}\int_G|\grad \varphi|^2 + \frac{C}{a} \int_G \psi \;\;\;\; \text{ for all } a > 0
\]
and 
\[
	\int_G \psi \ln \left(\frac{\;\psi\;}{ \overline{\psi}}\right) \leq \frac{1}{\beta_0}\left\{ \int_G \psi \right\}\int_G |\grad \ln(\psi)|^2 + C\int_G\psi
\]
with $G$ being a finitely connected two-dimensional domain with a smooth boundary, $\varphi, \psi$ being sufficiently regular functions such that all integrals are defined, $\beta_0$ as well as $C$ being fixed constants and $\overline{\psi} \defs \frac{1}{|G|}\int_G \psi$ (cf.\ \cite{MyExistence}, \cite{WinklerSmallMass}). Notably for trivial examples of $\psi$ (e.g.\ constant functions), it is easy to see that the above inequalities hold without the additional mass term. Thus keeping in mind the often striking results at the limits of the original Trudinger--Moser inequality, it seems potentially fruitful to investigate the degree to which we can minimize the constant $C$ or if it is possible to even remove the (potentially vestigial) mass term altogether. As the following result shows, the latter is in fact achievable (at the cost of potentially smaller value of $\beta_0$) and, as we will see later in this paper, in fact conducive to improving our understanding of the aforementioned chemotaxis-Navier--Stokes systems.
\begin{theorem}
	\label{theorem:func_ineq}
	For any bounded, finitely connected domain $G \subseteq \R^2$ with $C^2$-boundary, there exists a constant $\beta_0 > 0$ such that
	\begin{equation}
	\int_G \varphi (\psi - \overline{\psi}) \leq \frac{1}{a}\int_G \psi \ln \left(\frac{\;\psi\;}{ \overline{\psi}}\right) + \frac{a}{4\beta_0} \left\{ \int_G \psi \right\}\int_G|\grad \varphi|^2
	\label{eq:func_ineq1}
	\end{equation}
	for all $\varphi \in W^{1,2}(G)$, positive $\psi \in L^p(G)$ with $p > 1$ and $a > 0$ and
	\begin{equation}
	\int_G \psi \ln \left(\frac{\;\psi\;}{ \overline{\psi}}\right) \leq \frac{1}{\beta_0}\left\{ \int_G \psi \right\}\int_G |\grad \ln(\psi)|^2
	\label{eq:func_ineq2}
	\end{equation}
	for all positive $\psi \in L^p(G)$ with $p > 1$ and $\ln(\psi) \in W^{1,2}(G)$. Here, $\overline{\psi} \defs \frac{1}{|G|} \int_G \psi$.
\end{theorem} \noindent
\begin{remark}
	For functions $\psi$ of higher regularity (e.g.\ $C^1(\overline{G})$), the inequality in (\ref{eq:func_ineq2}) can be rewritten as
	\[
		\int_G \psi \ln \left(\frac{\;\psi\;}{ \overline{\psi}}\right) \leq \frac{2}{\beta_0}\left\{ \int_G \psi \right\}\int_G \frac{|\grad \psi|^2}{\psi^2},
	\]
	which is how we will use it in \Cref{section:application}.
\end{remark}
\paragraph{Key ideas. } Both (\ref{eq:func_ineq1}) and (\ref{eq:func_ineq2}) are ultimately a consequence of a corollary to the Trudinger--Moser inequality (cf.\ \cite{moser1971sharp}, \cite{trudinger1967imbeddings}), namely the inequality
\[
	\int_G e^{\beta\varphi} \leq C_G \exp\left( \frac{1}{4\beta} \int_G |\grad \varphi|^2 + \frac{1}{|G|} \int_G \varphi \right) \;\;\;\;\;\; \text{ for all } \varphi \in W^{1,2}(G)  \numberthis \label{eq:intro_mt}
\]
with appropriate $\beta > 0$ and $C_G \geq |G|$ (This lower bound for $C_G$ directly follows from setting $\varphi \defs 0$).
Although the above inequality is fairly easy to derive if optimality of the constants is not necessarily an objective, our aim here will be minimizing the constant $C_G$ as this is central to the derivation of our new inequalities. In fact, the ideas that make this possible are arguably the linchpin to this whole paper and seem nonetheless not widely explored in this context.
Instead prior efforts seem to mostly focus on maximizing $\beta$ in various settings (cf.\ \cite{MoserTrudingerOpt1}, \cite{MoserTrudingerOpt2}, \cite{moser1971sharp},  for instance), while the key to our result is in fact sacrificing the size of $\beta$ in favor of smaller $C_G$. To our knowledge, minimizing $C_G$ has been thus far only considered on the spheres $\mathbb{S}^n$ and in related settings (cf.\ \cite{MR3377875}, \cite{MR677001}, \cite{Xiong2018}, for instance).
\\[0.5em]
To achieve such a minimization of $C_G$ in planar domains then, we begin by employing techniques from the calculus of variations to first find a minimizer $\varphi_\beta$ of the functional 
\[
	J_\beta(\varphi) \defs \frac{1}{4\beta}\int_{G} |\grad \varphi|^2 + \frac{1}{|G|}\int_{G} \varphi - \ln\left(\frac{1}{|G|}\int_{G}e^\varphi\right) \;\;\;\; \text{ for all } \varphi \in W^{1,2}(G),
\]
which arises in a natural way from (\ref{eq:intro_mt}) after some rearrangement, for each $\beta \in (0,\pi]$.
\\[0.5em]
Having found such minimizers, we then show that they solve a Neumann problem corresponding to the equation
\[
	-\frac{1}{2\beta} \laplace \varphi_\beta  = -\frac{1}{|G|} + \frac{e^{\varphi_\beta}}{\int_G e^{\varphi_\beta}}
\]
in a weak sense. Using the regularity features of this Neumann problem combined with the regularity properties of the minimizers inherent to their construction, we can argue that all minimizers are bounded in $L^\infty(G)$ independent of $\beta$. We further note that, when $\beta$ becomes small, the Laplacian on the left-hand side of the above equation becomes arbitrarily strong, which has the following consequence: If we restrict ourselves to solutions, which are normalized to $\int_\Omega \varphi_\beta = 0$ and are bounded in $L^\infty(G)$ by a fixed constant independent of $\beta$, then as $\beta \searrow 0$ the only solution that fulfills these constraints is $\varphi_\beta  \equiv 0$.
Combining these insights, we can then conclude that our minimizers must be equal to zero everywhere as well if $\beta$ is sufficiently small.
But this directly gives us $J_\beta \geq J_\beta(\varphi_\beta) = J_\beta(0) = 0$ for sufficiently small $\beta$, which implies (\ref{eq:intro_mt}) with $C_G = |G|$.
\\[0.5em]
Knowing that (\ref{eq:intro_mt}) is in fact true with $C_G = |G|$ then allows us to derive our functional inequalities in a similar fashion to \cite{MyExistence} and \cite{WinklerSmallMass}. 
  
\paragraph{A chemotaxis-Navier--Stokes system.}
The above result was not created in a vacuum but rather during the study of the long time behavior of the system 
\[
\left\{\;\;
\begin{aligned}
n_t + u \cdot \grad n &\;\;=\;\; \laplace n - \div (nS(x,n,c) \cdot \grad c), \\
c_t + u\cdot \grad c &\;\;=\;\; \laplace c - n f(c), \\ 
u_t + (u\cdot \grad) u  &\;\;=\;\; \laplace u + \grad P + n \grad \phi, \;\;\;\;\;\; \nabla \cdot u = 0,
\end{aligned}
\right.
\numberthis \label{problem}
\]
of partial differential equations arising from biology in bounded two-dimensional domains with a smooth boundary. It models chemotaxis, the directed movement of cells along a chemical gradient toward an attractant, under the influence of a surrounding fluid. Here, $n$ represents the cell population, $c$ represents an attractant concentration, $u$ and $P$ represent the fluid velocity field and associated pressure, respectively.
\\[0.5em]
Systems of this type, though without fluid interaction, were first introduced in the seminal work \cite{keller1970initiation} by Keller and Segel in 1970 and have since been developed in several directions. One such direction stems from the observation by Dombrowski et al.\ (cf.\ \cite{PhysRevLett.93.098103}) that a population of \emph{Bacillus subtilis} generate speeds of fluid movement after aggregation that seemed to be insufficiently explained by considering each cell in isolation. This challenges the standing assumption that fluid-cell interaction can be disregarded because each cell has only negligible influence on the fluid. As such, Tuval et al.\  \cite{Tuval2277} introduced the system (\ref{problem}) containing a full Navier--Stokes equation to model the fluid interaction, which we present here somewhat normalized. The key interactions are the convective forces of the fluid acting on the cells and attractant modeled by the terms $u\cdot \grad n$ and $u \cdot \grad c$, respectively, and the buoyant forces of the cells acting on the fluid modeled by the term $n\grad\phi$.
\\[0.5em]
If we assume the chemotactic sensitivity $S$ to be scalar, systems of this type are fairly well understood, which is often due to convenient energy-type structures. In fact, there have been various results discussing well-posedness in the whole space case under varying assumption on the parameter functions $S$, $f$, $\phi$ and initial data (cf.\ \cite{MR3208807}, \cite{MR2754058}) as well as results about the global existence of unique classical solutions in bounded two-dimensional domains (cf.\ \cite{WinklerExistence}). In bounded three-dimensional domains, there are generally only less ambitious existence results available, likely due to the problematic Navier--Stokes equation (cf.\ \cite{WinklerExistence}, \cite{MR3542616}, \cite{MR3605965}). For a broader overlook about many types of chemotaxis problems and results concerning them, we refer the reader to the survey \cite{MR3351175}.
\paragraph{Eventual smoothness of solutions.}
Before we formulate the second main result of this paper, let us first establish the full context, in which we want to analyze the system (\ref{problem}):
\\[0.5em]
We let $\Omega \subseteq \R^2$ be a bounded, convex domain with a smooth boundary. We then add the boundary conditions
\begin{equation}
\grad n \cdot \nu = n(S(x,n,c)\grad c)  \cdot \nu, \;\; \grad c \cdot \nu = 0, \;\; u = 0 \;\; \text{ for all } x \in \partial\Omega, t > 0 \label{boundary_conditions}
\end{equation}
and initial conditions
\begin{equation}
n(x, 0) = n_0(x),\;\;c(x,0) = c_0(x),\;\;u(x,0) = u_0(x) \;\; \text{ for all } x\in\Omega
\label{initial_data}
\end{equation}
for initial values with the properties
\begin{equation}
\left\{\;
\begin{aligned}
n_0 &\in C^{\iota}(\overline{\Omega})  && \text{ for some } \iota > 0\text{ and with } n_0 > 0 \text{ in } \overline{\Omega}, \\
c_0 &\in W^{1,\infty}(\Omega) &&\text{ with } c_0 > 0 \text{ in } \Omega,  \\
u_0 &\in D(A_2^\vartheta)&& \text{ for some } \vartheta \in (\tfrac{1}{2}, 1)
\end{aligned}
\right.
\label{initial_data_props}
\end{equation}
to (\ref{problem}). Here, $A_2$ denotes the Stokes operator on the Hilbert space $L^2_\sigma(\Omega) \defs \{ \varphi \in (L^2(\Omega))^2 \;|\; \div \varphi = 0 \}$ of all solenoidal functions in $(L^2(\Omega))^2$. For more details concerning this space and operator, see \Cref{section:application_final_estimates}.
\\[0.5em]
For the functions $f, S$ and $\phi$ that parameterize (\ref{problem}), we will throughout this paper assume that
\begin{equation}
f \in C^1([0,\infty)) \stext{with} f(0) = 0 \stext{and} f(x) > 0 \;\; \text{ for all } x\in(0,\infty),
\label{f_regularity}
\end{equation}
that, for $S = (S_{ij})_{i,j \in\{1,2\}}$,
\begin{equation}
S_{ij} \in C^2(\overline{\Omega}\times[0,\infty)\times[0,\infty)) \;\;\text{ for } i,j \in \{1,2\}
\label{S_regularity},
\end{equation}
that
\begin{equation}
|S(x,n,c)| \leq S_0(c) \;\;\;\; \text{ for all } (x,n,c) \in \overline{\Omega}\times[0,\infty)^2 \;\; \text{ and some nondecreasing } S_0: [0,\infty) \rightarrow [0, \infty)
\label{S_0_bound}
\end{equation}
and that
\begin{equation}
\phi \in W^{2,\infty}(\Omega) \label{phi_regularity}.
\end{equation}
Given this setting, let us now first cite the following existence result for global generalized mass-preserving solutions from \cite{MyExistence}:
\begin{theoremext}
	\label{theorem:ext_existence}
	Let $\Omega \subseteq \R^2$ be a bounded, convex domain with a smooth boundary. If we then assume that $f$, $S$, $\phi$ satisfy (\ref{f_regularity})--(\ref{phi_regularity}) and the initial data have the properties outlined in (\ref{initial_data_props}), then the system (\ref{problem}) with initial data and boundary conditions (\ref{boundary_conditions}) and (\ref{initial_data}) has a global mass-preserving generalized solution $(n,c,u)$ in the sense of \Cref{definition:weak_solution} below.	
\end{theoremext}
\noindent It is for these generalized solutions that we prove the following eventual smoothness and stabilization result as we will make extensive use of the fact that they are the limit of a certain sequence of approximate solutions:
\begin{theorem}
	\label{theorem:eventual_smoothness}
	Let $\Omega \subseteq \R^2$ be a bounded, convex domain with a smooth boundary. Assume further that $f$, $S$, $\phi$ satisfy (\ref{f_regularity})--(\ref{phi_regularity}) and the initial data $(n_0, c_0, u_0)$ have the properties outlined in (\ref{initial_data_props}). Then for the generalized mass-preserving solution $(n,c,u)$ of (\ref{problem}) with (\ref{boundary_conditions}) and (\ref{initial_data}) constructed in \Cref{theorem:ext_existence}, there exists a time $t_0 > 0$ such that
	\[
		(n,c,u) \in C^{2,1}(\overline{\Omega} \times [t_0,\infty)) \times C^{2,1}(\overline{\Omega} \times [t_0,\infty)) \times C^{2,1}(\overline{\Omega} \times [t_0,\infty); \R^2).
	\]
	Further, there exists $P \in C^{1,0}(\overline{\Omega}\times[t_0,\infty))$ such that $(n,c,u,P)$ is a classical solution of (\ref{problem}) on $\Omega \times (t_0,\infty)$ with boundary conditions (\ref{boundary_conditions}). 
	\\[0.5em]
	Additionally,
	\begin{equation}
		n(\cdot, t) \rightarrow \frac{1}{|\Omega|}\int_\Omega n_0, \;\;\;\; c(\cdot, t) \rightarrow 0, \;\;\;\; u(\cdot, t) \rightarrow 0 \label{eq:eventual_convergence} 
	\end{equation}
	in $C^2(\overline{\Omega})$ or $C^2(\overline{\Omega};\R^2)$, respectively, as $t \rightarrow \infty$.
\end{theorem}

\paragraph{Complications.}
As already expanded upon in the related existence theory in \cite{MyExistence}, the two key features of (\ref{problem}) that complicate any analysis of the system are that we allow for general matrix valued sensitivities $S$ and use a full Navier--Stokes equation as the fluid model. Both are mainly problematic because they restrict our access to good, immediately available a priori information we can use as a baseline for later arguments.
\\[0.5em]
For scalar sensitivities $S$, there exist many results about similar systems to (\ref{problem}), with or without fluid interaction, concerning global existence (cf.\ \cite{MR2754058}, \cite{WinklerExistence}, \cite{MR3369260}) and long time behavior (cf.\ \cite{MR3149063}, \cite{MR3605965}) due to some very convenient energy inequalities. In the matrix-valued case, these energy inequalities are no longer available. This makes analysis of especially the first equation in (\ref{problem}) highly difficult. Therefore to our knowledge, prior work concerning non-scalar sensitivities has either hinged on some strong assumptions about $S$ or the initial data (cf.\ \cite{MR3562314}, \cite{MR3531759}, \cite{MR3801284}, \cite{MR3401606}, \cite{MR3542964}), on adding sufficiently strong nonlinear diffusion to the first equation (cf.\ \cite{MR3426095}) or only constructing generalized solutions (cf.\ \cite{MyExistence}, \cite{WinklerStokesCase}). Even in the fluid-free version of (\ref{problem}) without imposing any strong assumptions on $S$, global smooth solutions in the two-dimensional case seem to have thus far only been constructed under significant smallness conditions for $c_0$ (cf.\ \cite{MR3302296}) and, if we allow for general initial data and space dimension, only global generalized solutions (similar to those in \Cref{definition:weak_solution}) seem to be available (cf.\ \cite{WinklerLargeDataGeneralized}).
\\[0.5em]
Matrix valued sensitivities were introduced because they are of significant interest from a modeling standpoint. In models with scalar $S$, it been shown that solutions homogenize over time, which does not agree with the structure formation observed in experiments (cf.\ \cite{MR3149063}). As newly formed structures tend to originate at the boundaries (cf.\ \cite{PhysRevLett.93.098103}), modern modeling approaches introduce rotational flux components near said boundaries, leading to a sensitivity function $S$ that looks somewhat like
\[
	S = a\left(\,\begin{matrix}1 & 0 \\ 0 & 1\end{matrix}\, \right) + b \left(\,\begin{matrix}0 & -1 \\ 1& 0\end{matrix} \,\right) \;\;\;\; \text{ for }a > 0, \; b \in \R
\]
with significant non-diagonal entries (cf.\ \cite{MR3294964}, \cite{MR2505083}).
\\[0.5em]
The second complication inhibiting our access to a priori information is of course the famously hard to handle Navier--Stokes equation modeling the fluid. If we remove the nonlinear convection term and simplify the fluid model to a Stokes equation, a similar result about the eventual smoothness of generalized solutions not unlike those discussed here can be found in \cite{WinklerEventualSmoothness}. Sadly the methods seen there in large do not translate to the full Navier--Stokes case.
\\[0.5em]
As we will see in \Cref{section:application} or, more specifically, the proofs of \Cref{lemma:nlnn_integrabilty}, \Cref{lemma:u_props} and \Cref{lemma:c_stabilization}, these complications as far as we know can only be overcome (at least if we do not want to pose strong restrictions on our parameters) due to our functional inequalities in \Cref{theorem:func_ineq}, which in our opinion certainly underlines their significant usefulness.
\paragraph{Key ideas.}
As we only consider the generalized solutions constructed in  \Cref{theorem:ext_existence}, we will naturally require some of their specific structure for our proof of their eventual smoothness, namely that they are the limit of certain approximate solutions $(n_\eps, c_\eps, u_\eps)_{\eps \in (0,1)}$. The key idea then is to show that the approximate solutions are uniformly bounded in some sufficiently strong parabolic Hölder spaces from some time $t_0 > 0$ onward and use compact embedding properties of such spaces to show that the limit functions $(n,c,u)$ posses a similarly high level of regularity. We then only need to further derive that $n$ eventually fulfills a weak solution property of the same kind as the ones for $c$ and $u$ in \Cref{definition:weak_solution} as this allows us to use standard parabolic regularity theory to argue that the generalized solutions from \Cref{theorem:ext_existence} become in fact classical. 
\\[0.5em]
To do this, we need to extract significantly stronger a priori estimates for the approximate solutions (albeit maybe only after some time has passed) than for the existence theory in \cite{MyExistence} and this is naturally where our new functional inequalities come in. Our arguments will be based on the initially fairly weak regularity information
\[
	\int_0^\infty \int_\Omega \frac{|\grad n_\eps|^2}{n_\eps^2} \leq C \;\;\;\; \text{ for all } \eps \in (0,1)
\]
derived in \Cref{lemma:basic_props}, which leads to the following two important global integrability properties (cf.\ \Cref{lemma:nlnn_integrabilty}) due to the new functional inequalities (\ref{eq:func_ineq1}) and (\ref{eq:func_ineq2}):
\[
	\int_0^\infty\int_\Omega n_\eps\ln\left( \frac{n_\eps}{\overline{n_0}} \right) \leq C \stext{ and } \int_0^\infty\int_\Omega n_\eps \grad \phi \cdot u_\eps \leq C \;\;\;\; \text{ for all } \eps \in (0,1).
\]
The former then allows us to show in \Cref{lemma:c_stabilization} and \Cref{lemma:grad_c_long_time_small} that $\|c_\eps(\cdot, t)\|_\L{\infty}$ and
\[
\int_t^\infty \int_\Omega |\grad c_\eps|^2
\] become uniformly small as $t \rightarrow \infty$ while the latter allows us in \Cref{lemma:u_props} to derive that
\[
	\int_0^\infty \int_\Omega |u_\eps|^2  \leq C \;\; \text{ and } \;\; \int_0^\infty \int_\Omega |\grad u_\eps|^2  \leq C \;\;\;\; \text{ for all } \eps \in (0,1).
\] 
Both proofs again heavily rely on \Cref{theorem:func_ineq}.
\\[0.5em] 
Albeit in a weak sense, the above statements already suggest that the functional
\[
	\mathcal{F}_\eps(t) \defs \int_{\Omega} n_\eps \ln\left( \frac{n_\eps}{\overline{n_0}}\right) + \frac{1}{2}\int_{\Omega} |\grad c_\eps|^2 + \frac{1}{2C}\int_{\Omega} |u_\eps|^2
\] 
introduced in \Cref{section:application_final_estimates} might become uniformly small for large times $t$. That this is in fact the case is shown in \Cref{lemma:small_things_stay_small} and \Cref{lemma:things_are_small} by first arguing that the functional is small at some time $t_\eps$ for each $\eps \in (0,1)$ prior to a time $t$ independent of $\eps$ based on the integrability properties above and then deriving a differential inequality for $\mathcal{F}_\eps$ via testing methods to show that, if $\mathcal{F}_\eps$ ever gets small enough, it in fact stays small. From this argument, we additionally gain certain useful integrability properties for some higher order terms.  
\\[0.5em]
This now not only already gives us fairly strong stabilization properties, but also takes us over the critical point in terms of regularity such that the standard bootstrap techniques seen from \Cref{section:bootstrap} onward will take us all the way to our desired result.
\section{Two new functional inequalities based on the Trudinger--Moser inequality}
For the purposes of this section, $G \subseteq \R^2$ is always a finitely connected, bounded domain (cf.\ \cite{chang1988conformal}) with a $C^2$-boundary.
\subsection{The Trudinger--Moser inequality}
We will start the derivation of our new functional inequalities by reminding ourselves of an already well-known inequality first pioneered by Trudinger in \cite{trudinger1967imbeddings} and then later refined by Moser in \cite[Theorem 1]{moser1971sharp}, which will serve as the starting point for all further considerations. As it is somewhat more convenient for our purposes though, we use the more recent formulation of the same inequality by Chang in \cite[Proposition 2.3]{chang1988conformal}, which can be extended from $C^1(\overline{G})$ to $W^{1,2}(G)$ by a straightforward density argument: 
\begin{theorem}
	Let $G \subseteq \R^2$ be a finitely connected, bounded domain (cf.\ \cite{chang1988conformal}) with a $C^2$-boundary. Then there exists a constant $C_G \geq |G|$ such that, for all $\varphi \in W^{1,2}(G)$ 
	with 
	\begin{equation*}
	\int_{G} |\grad \varphi|^2 \leq 1 \text{ and } \int_{G} \varphi = 0 
	\end{equation*}
	and $0 < \beta \leq 2\pi$, we have
	\begin{equation*}
		\int_{G} e^{\beta \varphi^2} \leq C_G.
	\end{equation*}
	\label{theorem:moser-trudinger-raw}
\end{theorem}\noindent
As the above restrictions on $\varphi$ can be somewhat inconvenient, we will now prove a standard corollary of the Trudinger--Moser inequality eliminating said restrictions at the cost of some corresponding terms on the right and a slightly different term on the left of the inequality:
\begin{corollary}
	\label{corollary:moser-trudinger-raw}
	For each $0 < \beta \leq 2\pi$ and $\varphi \in W^{1,2}(G)$, we have 
	\begin{equation}
	\int_{G} e^\varphi \leq C_G \exp\left( \frac{1}{4\beta}\int_{G} |\grad \varphi|^2  + \frac{1}{|G|}\int_{G} \varphi \right)
	\label{eq:cor:moser-trudinger-raw}
	\end{equation}
	with $C_G$ from \Cref{theorem:moser-trudinger-raw}.
\end{corollary}
\begin{proof}
	As (\ref{eq:cor:moser-trudinger-raw}) is trivially true for constant functions $\varphi$ with $C_G = |G|$, we can assume that $\|\grad \varphi\|_\L{2} > 0$ for the remainder of this proof without loss of generality. Then by using Young's inequality to see that
	\[
		\varphi - \overline{\varphi} \leq |\varphi - \overline{\varphi}| \leq \beta\left(\frac{\varphi - \overline{\varphi}}{\|\grad \varphi\|_{L^2(G)}}\right)^2 + \frac{1}{4\beta} \|\grad \varphi\|^2_{L^2(G)}
	\]
	with $\overline{\varphi} \defs \frac{1}{|G|} \int_G \varphi$,  we directly gain from \Cref{theorem:moser-trudinger-raw} that
	\[
		\int_G e^{\varphi-\overline{\varphi}} \leq C_G\exp\left( \frac{1}{4\beta} \int_G |\grad \varphi|^2 \right)
	\]
	or further that
	\[
		\int_G e^{\varphi} \leq C_G\exp\left( \frac{1}{4\beta} \int_G |\grad \varphi|^2 + \frac{1}{|G|}\int_G \varphi \right) 
	\]
	for all $0 < \beta \leq 2\pi$ and $\varphi \in W^{1,2}(G)$.
\end{proof}
\noindent As integrals of the form $\int_G e^\varphi$ will naturally play a significant role in the following arguments, let us briefly note that the above corollary ensures that said integrals are always finite and positive if $\varphi \in W^{1,2}(G)$, which makes them reasonably straightforward to handle. 
\subsection{A variational approach to minimizing $C_G$}
Understanding the relationship of the constants $C_G$ and $\beta$ in \Cref{corollary:moser-trudinger-raw} will be the linchpin to our proof of \Cref{theorem:func_ineq}. While there have been considerable efforts to achieve the above inequality for optimal, meaning large, values of $\beta$ in many different contexts as laid out in the introduction, we will be more interested in how small we can make $C_G \geq |G|$ at the cost of only allowing for smaller values of $\beta$, which is not as widely studied. 
\\[0.5em]
Therefore, what we are now essentially looking at is a minimization problem, which we will handle using variational methods. Concerning which functional to minimize, we let ourselves be guided by a similar approach in \cite[Theorem 18.2.1]{MR3052352} to minimizing the constant $C_G$ on the sphere $\mathbb{S}^2$ (cf.\ \cite{MR3377875} for an overview about proof techniques on $\mathbb{S}^2$). Thus for each $\beta \in (0,2\pi]$, we will analyze the following functional:
\begin{equation}\label{eq:J_functional}
	J_\beta(\varphi) \defs \frac{1}{4\beta}\int_{G} |\grad \varphi|^2 + \frac{1}{|G|}\int_{G} \varphi - \ln\left(\frac{1}{|G|}\int_{G}e^\varphi\right) \;\;\;\; \text{ for all } \varphi \in W^{1,2}(G).
\end{equation}
As $J_\beta \geq 0$ immediately implies that (\ref{eq:cor:moser-trudinger-raw}) holds with $C_G = |G|$, which is the smallest possible value for $C_G$ in said inequality, it will be our aim for the remainder of this section to show that minimizers $\varphi_\beta$ for $J_\beta$ exist and that, for sufficiently small $\beta$, they have the property $J_\beta(\varphi_\beta) = 0$.
\\[0.5em]
To do this, let us now first consider a basic lower boundedness and coerciveness property of $J_\beta$ directly following from \Cref{corollary:moser-trudinger-raw}:
\begin{lemma}
	There exists a constant $C \geq 0$ such that
	\[
	J_\beta(\varphi) \geq \frac{1}{8\beta} \int_{G} |\grad \varphi|^2 - C \geq -C
	\]
	for all $\varphi \in W^{1,2}(G)$ and $\beta \in (0,\pi]$.
	\label{lemma:J_coercive}
\end{lemma}
\begin{proof}
	Let $\beta \in (0,\pi]$ and $\gamma \defs 2\beta \in (\beta,2\pi]$. Then we know from \Cref{corollary:moser-trudinger-raw} that
	\[
	- \ln\left(\frac{1}{|G|}\int_{G} e^\varphi \right) \geq 
	-\frac{1}{4\gamma}\int_{G} |\grad \varphi|^2  - \frac{1}{|G|}\int_{G} \varphi - \ln\left(\frac{C_G}{|G|}\right).
	\]
	for all $\varphi \in W^{1,2}(G)$ after some minor rearranging.
	If we now apply this to $J_\beta$, we see that
	\begin{align*}
	J_\beta(\varphi) 
	\geq& \left(\frac{1}{4\beta} - \frac{1}{4\gamma} \right) \int_G |\grad \varphi|^2 - \ln\left(\frac{C_G}{|G|}\right) \\
	=& \frac{1}{8\beta} \int_G |\grad \varphi|^2 - \ln\left(\frac{C_G}{|G|}\right)
	\end{align*}
	for all $\varphi \in W^{1,2}(G)$, which completes the proof.
\end{proof} \noindent
This property will now enable us to find a minimizer $\varphi_\beta$ for each $J_\beta$ by first allowing us to construct a minimizing sequence for each functional and then arguing that said sequences converge in certain topologies to some limit function in $W^{1,2}(G)$. We then only need to further show that said convergence properties lead to sufficient estimates to ensure that the limit object is in fact an actual minimizer.
\\[0.5em]
Moreover by utilizing the now established minimizer property, we will additionally show that each $\varphi_\beta$ solves a certain weak elliptic Neumann boundary value problem as a first step in our efforts to show that $J_\beta(\varphi_\beta) = 0$. 
\begin{lemma}
	\label{lemma:minimizer}
	For each $\beta \in (0,\pi]$, there exists a function $\varphi_\beta \in W^{1,2}(G)$ with
	\[
	\int_{G} \varphi_\beta = 0
	\]
	and
	\[
	\frac{1}{2\beta}\int_{G} \grad \varphi_\beta \cdot \grad \psi = - \frac{1}{|G|} \int_{G} \psi + \frac{\int_{G} \psi e^{\varphi_\beta}}{\int_{G} e^{\varphi_\beta}} \numberthis \label{eq:minimizer-weak-neumann}
	\]
	for all $\psi \in W^{1,2}(G)$, which is a minimizer of $J_\beta$,
	meaning that
	\[
		\inf_{\varphi \in W^{1,2}(G)} J_\beta(\varphi) = J_\beta(\varphi_\beta).
	\] 
\end{lemma}
\begin{proof}
	We fix $\beta \in (0,\pi]$. We know from \Cref{lemma:J_coercive} that $\inf_{\varphi \in W^{1,2}(G)} J_\beta(\varphi) \geq -C$ for some $C > 0$ and we can therefore choose a (minimizing) sequence $(\varphi_k)_{k\in\N} \subseteq W^{1,2}(G)$ such that
	\[
	J_\beta(\varphi_k) \rightarrow \inf_{\varphi \in W^{1,2}(G)} J_\beta(\varphi)
	\]
	as $k \rightarrow \infty$. Without loss of generality, we can further assume that 
	\[
	\int_{G} \varphi_k = 0 \;\;\;\; \text{ for all } k\in\N
	\]
	because it is easily seen that $J_\beta$ is invariant under the addition of constants to its argument.
	Because the sequence $(J_\beta(\varphi_k))_{k\in\N}$ converges, it is bounded and thus \Cref{lemma:J_coercive} implies that the sequence \[
	\left(\int_{G} |\grad \varphi_k|^2\right)_{k\in\N}
	\]
	is bounded as well. As we know that $\int_{G} \varphi_k = 0$ for all $k\in\N$, the Poincaré inequality (cf.\ \cite[p. 312]{BrezisFAandPDE}) implies that therefore the sequence $(\varphi_k)_{k\in\N}$ is bounded in $W^{1,2}(G)$ as well. Without loss of generality (by choosing fitting subsequences), this allows us to assume that there exists a function $\varphi_\beta \in W^{1,2}(G)$ with
	\begin{equation}\label{eq:convergence_minimizer}
		\left\{ \begin{aligned}
			\varphi_k &\rightharpoonup \varphi_\beta && \text{ in } W^{1,2}(G) \\
			\varphi_k &\rightarrow \varphi_\beta && \text{ in } L^{1}(G) \text{ and } L^{2}(G)
		\end{aligned}
		\right.
	\end{equation}
	as $k \rightarrow \infty$ by standard compactness arguments. The above $L^1(G)$ convergence then ensures that $\int_G \varphi_\beta = 0$.
	Further due to the mean value theorem, we can now observe that
	\begin{align*}
	\left| \int_{G} e^{\varphi_k} - \int_{G} e^{\varphi_\beta} \right| \leq \int_{G} |e^{\varphi_k} - e^{\varphi_\beta}| \leq \int_{G} |\varphi_k - \varphi_\beta| e^{|\varphi_k| + |\varphi_\beta|} 
	\leq \|\varphi_k - \varphi_\beta\|_\L{2} \left(\int_{G} e^{2|\varphi_k| + 2|\varphi_\beta|}\right)^{\frac{1}{2}}
	\end{align*}
	for all $k\in\N$. By the $L^2(G)$ convergence from (\ref{eq:convergence_minimizer}) and using the fact that $\int_{G} e^{2|\varphi_k| + 2|\varphi_\beta|}$ is uniformly bounded due to \Cref{corollary:moser-trudinger-raw} and the $W^{1,2}(G)$ bound for the sequence already established prior, this directly implies that 
	\[
		\int_{G} e^{\varphi_k} \rightarrow \int_{G} e^{\varphi_\beta}
	\]
	as $k \rightarrow \infty$.
	\\[0.5em]
	Using this convergence property combined with (\ref{eq:convergence_minimizer}), we then see that
	\begin{align*}
	\inf_{\varphi \in W^{1,2}(G)} J_\beta(\varphi)= \lim_{k\rightarrow \infty} J_\beta(\varphi_k)  = \frac{1}{4\beta} \liminf_{k\rightarrow \infty} \int_{G} |\grad \varphi_k|^2 + \frac{1}{|G|}\lim_{k\rightarrow \infty} \int_{G} \varphi_k - \ln \left(\frac{1}{|G|}\lim_{k\rightarrow \infty} \int_{G} e^{\varphi_k}\right) 
	\geq J_\beta(\varphi_\beta)
	\end{align*}
	and therefore that
	\[
	J_\beta(\varphi_\beta) = \inf_{\varphi \in W^{1,2}(G)} J_\beta(\varphi).
	\]
	Thus, $\varphi_\beta$ is a minimizer of $J_\beta$.
	\\[0.5em]
	It now only remains to show that $\varphi_\beta$ is also a weak solution of the Neumann problem corresponding to (\ref{eq:minimizer-weak-neumann}). To this end, let now $\psi \in W^{1,2}(G)$ be fixed, but arbitrary. We then consider the function
	\[
	f \colon (-1,1) \rightarrow \R, \;\;\;\; t \mapsto J_\beta(\varphi_\beta + t\psi),
	\]
	which has a global minimum in $0$ by our observations about $J_\beta$.
	Then
	\begin{align*}
	f(t) = \frac{1}{4\beta} \int_{G} |\grad \varphi_\beta|^2 + t \frac{1}{2\beta} \int_{G} \grad \varphi_\beta \cdot \grad \psi + t^2 \frac{1}{4\beta}\int_{G} |\grad \psi|^2 
	+ \frac{1}{|G|} \int_G \varphi_\beta + t\frac{1}{|G|}\int_{G} \psi - \ln\left( \frac{1}{|G|}\int_{G} e^{\varphi_\beta + t\psi}\right).
	\end{align*}
	One easily sees that $f$ is differentiable as it is mostly a polynomial in $t$ and the remaining terms are amenable to results about the differentiation of parameter integrals (Note that \Cref{corollary:moser-trudinger-raw} can be used to establish the necessary integrability properties). The minimality property of $f$ in $0$ therefore implies that
	\[
	0 = f'(0) = \frac{1}{2\beta}\int_{G}\grad \varphi_\beta \cdot \grad \psi + \frac{1}{|G|} \int_{G}\psi - \frac{\int_G \psi e^{\varphi_\beta}}{\int_{G}e^{\varphi_\beta}},
	\]
	which gives us (\ref{eq:minimizer-weak-neumann}) and thus completes the proof.
\end{proof}
\noindent Having now constructed the minimizers $\varphi_\beta$, the key to showing that for sufficiently small $\beta$ we have $J_\beta \geq J_\beta(\varphi_\beta) = 0$ is understanding the weak elliptic Neumann problem 
\[
	\left\{ \begin{aligned} 
		-\tfrac{1}{2\beta} \laplace \varphi_\beta  &= -\tfrac{1}{|G|} + \tfrac{e^{\varphi_\beta}}{\int_G e^{\varphi_\beta}} \;\;\;\;\;\;&&\text{ on }G,\\
		\grad \varphi_\beta \cdot \nu &= 0 && \text{ on }\partial G.
		\end{aligned}
	\right.
\]
In this regard, the two most crucial insights about the above system as well as the minimizers are the following:
First, the minimizers are bounded in $W^{1,2}(G)$ independent of $\beta$ as a consequence of their minimization property, which by the Trudinger--Moser inequality as well as elliptic regularity properties of the system further results in an $L^\infty(G)$ bound for the minimizers, which is $\beta$-independent as well. Second by reducing the value of $\beta$, we can make the Laplacian in the above system arbitrarily strong when compared to the source terms on the right, which manifests as the following property: When only considering solutions that are normalized to $\int_G \varphi_\beta = 0$ and are bounded in $L^\infty(G)$ by some constant $C > 0$ independent of $\beta$, we can increase the strength of the Laplacian to such a degree that at some point the only member of the aforementioned solution class is $\varphi_\beta  \equiv 0$.
\\[0.5em]
Combined, this means that, for sufficiently small $\beta$, the minimizers $\varphi_\beta$ must be everywhere equal to zero as well. But this directly implies $J_\beta (\varphi_\beta) = J_\beta (0) = 0$. 
\\[0.5em]
We will now make these ideas precise to prove the following lemma.
\begin{lemma}
	There exists $\beta_0 \in (0,\pi]$ such that 
	\[
	\inf_{\varphi \in W^{1,2}(G)} J_{\beta}(\varphi) = 0
	\]
	for all $\beta \in (0, \beta_0]$.
	\label{lemma:minizer_0}
\end{lemma}
\begin{proof}
	For each $\beta \in (0,\pi]$, let $\varphi_\beta$ be the minimizer of $J_\beta$ found in \Cref{lemma:minimizer}. First note that there exists a constant $K_1 > 0$ such that
	\[
	0 = J_\beta(0) \geq \inf_{\varphi \in W^{1,2}(G)} J_\beta(\varphi) = J_\beta(\varphi_\beta) \geq \frac{1}{8\beta} \int_G |\grad \varphi_{\beta}|^2 - K_1
	\]
	and therefore 
	\[
	\int_G |\grad \varphi_{\beta}|^2 \leq 8\beta K_1 \leq 8\pi K_1
	\]
	for all $\beta \in (0,\pi]$ by \Cref{lemma:J_coercive}. 
	We now further observe that 
	\begin{equation}
	\int_{G} e^{\varphi_\beta} = \frac{|G|}{|G|} \int_{G} e^{\varphi_\beta}\geq|G| \exp\left(\frac{1}{|G|}\int_{G} \varphi_\beta \right) = |G| \label{eq:e_phi_bound}
	\end{equation}
	because of Jensen's inequality and the fact that $\int_{G} \varphi_\beta = 0$. Together, these two inequalities then give us that
	\[
	\left\|\frac{ e^{\varphi_\beta}}{\int_{G} e^{\varphi_\beta}}\right\|^2_{L^{2}(G)} \leq \frac{1}{|G|^2} \int_{G} e^{2\varphi_\beta} \leq \frac{C_G}{|G|^2} \exp\left( \frac{1}{2\pi}\int_{G} |\grad \varphi_\beta|^2\right)\leq
	\frac{C_G}{|G|^2} e^{4K_1} =: K_2 \numberthis \label{eq:minimizer-l2}
	\]
	for all $\beta \in (0,\pi]$ by way of \Cref{corollary:moser-trudinger-raw}. We further know that the functions $\varphi_\beta$ solve a weak Neumann problem in the sense seen in (\ref{eq:minimizer-weak-neumann}), which gives us that
	\[
		\int_{G} \grad \varphi_\beta \cdot \grad \psi = 2\beta\left[- \frac{1}{|G|} \int_{G} \psi + \frac{\int_{G} \psi e^{\varphi_\beta}}{\int_{G} e^{\varphi_\beta}}\right] 
	\]
	for all $\psi \in W^{1,2}(G)$ and $\beta \in (0,\pi]$. This and the fact that $\int_{G} \varphi_\beta = 0$ makes $\varphi_\beta$ accessible to standard elliptic regularity theory as e.g.\ found in Lemma 2 on page 217 of Reference \cite{MR601389}. Using said regularity results in combination with (\ref{eq:minimizer-l2}), we then gain a constant $K_3 > 0$ such that
	\begin{align*}
	\|\varphi_\beta\|_{W^{2,2}(G)} \leq& K_3 \left\| 2\beta\left[ -\frac{1}{|G|} + \frac{e^{\varphi_\beta}}{\int_{G} e^{\varphi_\beta}} \right] \right\|_\L{2} \\
	\leq & 2 K_3\beta \left[ \frac{1}{\sqrt{|G|}} + \left\|\frac{ e^{\varphi_\beta}}{\int_{G} e^{\varphi_\beta}}\right\|_\L{2} \right] \\
	\leq & 2 K_3 \pi\left[ \frac{1}{\sqrt{|G|}} +\sqrt{K_2} \right] =: K_4 \;\;\;\;\;\; \text{ for all } \beta \in (0,\pi].
	\end{align*}
	We can now further use the two-dimensional Sobolev inequality to find a constant $K_5 > 0$ with
	\[
	\|\varphi_\beta\|_\L{\infty} \leq K_5 \|\varphi_\beta\|_{W^{2,2}(G)} \leq K_4 K_5 =: K_6
	\]
	for all $\beta \in (0,\pi]$, meaning that the functions $\varphi_\beta$ are in fact uniformly bounded in $W^{2,2}(G)$ and $L^\infty(G)$. 
	\\[0.5em]
	We now set $\psi = \varphi_\beta$ in (\ref{eq:minimizer-weak-neumann}) to see that
	\begin{align*}\numberthis 
	\frac{1}{2\beta}\int_{G} |\grad \varphi_\beta|^2 = - \frac{1}{|G|} \int_{G} \varphi_\beta + \frac{\int_{G} \varphi_\beta e^{\varphi_\beta}}{\int_{G} e^{\varphi_\beta}} \;\;\;\; \text{ for all } \beta \in (0, \pi].
	\label{eq:minizer-test-with-self}  
	\end{align*}
	As a first consequence of (\ref{eq:minizer-test-with-self}) and the fact that $\int_G \varphi_\beta = 0$, we gain that
	\[
	\int_{G} \varphi_\beta e^{\varphi_\beta} \geq 0
	\] 
	and therefore that
	\[
	\frac{\int_{G} \varphi_\beta e^{\varphi_\beta}}{\int_{G} e^{\varphi_\beta}} \leq \frac{\int_{G} \varphi_\beta e^{\varphi_\beta}}{|G|}
	\]
	because of (\ref{eq:e_phi_bound}) for all $\beta \in (0,\pi]$. If we then apply this to (\ref{eq:minizer-test-with-self}), we gain that
	\begin{align*}
	\frac{1}{2\beta}\int_{G} |\grad \varphi_\beta|^2 &\leq \frac{1}{|G|} \left( \int_{G} \varphi_\beta ( e^{\varphi_\beta} - 1) \right) \leq \frac{1}{|G|} \left( \int_{G} |\varphi_\beta| | e^{\varphi_\beta} - e^0 |\right) \\
	&\leq \frac{1}{|G|} \int_{G} |\varphi_\beta|^2 e^{|\varphi_\beta|} \leq \frac{e^{K_6}}{|G|} \int_{G} |\varphi_\beta|^2 \leq \frac{e^{K_6}C_\text{p}^2}{|G|} \int_{G} |\grad \varphi_\beta|^2 \;\;\;\; \text{ for all } \beta \in (0, \pi]
	\end{align*}
	by the mean value theorem and the Poincaré inequality with constant $C_\text{p} > 0$. If $\beta$ is now smaller than or equal to 
	\[
	\beta_0 \defs \min\left(\frac{|G|}{4e^{K_6}C_\text{p}^2}, \pi \right),
	\]
	we gain 
	\[
	\int_{G} |\grad \varphi_\beta|^2 = 0
	\]
	from the previous inequality, which implies that $\varphi_\beta = 0$ as $\int_\Omega \varphi_\beta  = 0$. Therefore
	\[
	\inf_{\varphi \in W^{1,2}(G)} J_\beta(\varphi) = 
	J_\beta(\varphi_\beta) = J_\beta(0) = 0
	\]
	for all $\beta \in (0,\beta_0]$, which completes the proof.
\end{proof}\noindent
This new insight now allows us to significantly improve upon \Cref{corollary:moser-trudinger-raw} (along one specific axis) by just rearranging some terms in the functional $J_\beta$ defined in (\ref{eq:J_functional}) to gain the following:
\begin{corollary}
	For each $0 < \beta \leq \beta_0$ and $\varphi \in W^{1,2}(G)$, we have 
	\begin{equation}
	\int_{G} e^\varphi \leq |G| \exp\left( \frac{1}{4\beta}\int_{G} |\grad \varphi|^2  + \frac{1}{|G|}\int_{G} \varphi \right)
	\end{equation}
	with $\beta_0$ from \Cref{lemma:minizer_0}.
	\label{corollary:moser-trudinger-optimal}
\end{corollary}
\subsection{Proving our new functional inequalities}
After this brief excursion into the calculus of variations and the theory of elliptic problems, we will now refocus on proving \Cref{theorem:func_ineq} using our optimal \Cref{corollary:moser-trudinger-optimal} and similar methods as those seen in \cite{MyExistence} and \cite{WinklerSmallMass}:
\begin{proof}[Proof of \Cref{theorem:func_ineq}]
	Let $\varphi \in W^{1,2}(G)$, $\psi \in L^p(G)$ with $p > 1$, let $\psi$ be positive and let $m \defs \int_G \psi > 0$.
	Observe now for any $a > 0$ that
	\[
	\ln\left( \int_G e^{a\varphi} \right) = \ln\left( \int_G e^{a\varphi} \frac{m}{\psi} \frac{\psi}{m} \right) \geq \int_G \left( \ln(e^{a\varphi}) + \ln\left(\frac{m}{\psi}\right) \right)\frac{\psi}{m}= \frac{}{} \frac{a}{m}\int_G \varphi \psi - \frac{1}{m}\int_G \psi\ln\left(\frac{\psi}{m}\right) 
	\]
	by Jensen's inequality. Note that our choices of function spaces for $\varphi$ and $\psi$ ensure that all the integrals are well defined. If we now combine this with \Cref{corollary:moser-trudinger-optimal} (applied to $a\varphi$) and multiply by $\frac{m}{a}$, we get that
	\begin{align*}
		\int_G \varphi \psi &\leq \frac{1}{a}\int_G \psi\ln\left(\frac{\psi}{m}\right) + \frac{m}{a} \ln\left( |G| \exp\left( \frac{a^2}{4\beta_0} \int_\Omega |\grad \varphi|^2 + \frac{a}{|G|}\int_\Omega \varphi  \right)\right) \\
		&= \frac{1}{a}\int_G \psi\ln\left(\frac{\psi}{m}\right)  + \frac{am}{4\beta_0} \int_G |\grad \varphi|^2 + \frac{m}{|G|}  \int_G \varphi + \frac{m}{a}\ln(|G|) \\ 
		&= \frac{1}{a}\int_G \psi\ln\left(\frac{\psi}{\overline{\psi}}\right)  + \frac{a}{4\beta_0} \left\{\int_G \psi\right\}\int_G |\grad \varphi|^2 + \frac{m}{|G|}  \int_G \varphi
	\end{align*}
	or further that
	\[
	\int_G \varphi (\psi - \overline{\psi}) \leq \frac{1}{a} \int_G \psi \ln\left(\frac{\;\psi\;}{\overline{\psi}}\right) + \frac{a}{4\beta_0}\left\{\int_G \psi\right\} \int_G |\grad \varphi|^2 
	\]
	after some rearranging with $\overline{\psi} \defs \frac{1}{|G|}\int_G \psi$, which is (\ref{eq:func_ineq1}) exactly.
	\\[0.5em]
	For $\psi \in L^p(G)$ with $p > 1$, $\psi$ positive with $\ln(\psi) \in W^{1,2}(G)$, we now set \[
		\varphi \defs \ln\left(\frac{\;\psi\;}{\overline{\psi}}\right) \;\;\;\;\text{ and }\;\;\;\; a \defs 2
	\] 
	in the previous inequality to get that
	\[
	\int_G \psi \ln\left(\frac{\;\psi\;}{\overline{\psi}}\right) - \overline{\psi}\int_G \ln\left(\frac{\;\psi\;}{\overline{\psi}}\right) \leq \frac{1}{2} \int_G \psi \ln\left(\frac{\;\psi\;}{\overline{\psi}}\right) + \frac{1}{2\beta_0} \left\{\int_G \psi\right\} \int_G |\grad \ln(\psi)|^2 .
	\] 
	Because by Jensen's inequality we have \[
		\int_G \ln\left(\frac{\;\psi\;}{\overline{\psi}}\right) \leq |G| \ln\left( \frac{\;\frac{1}{|G|}\int_G \psi \;}{\overline{\psi}}\right) = |G|\ln(1) = 0,
	\]
	this directly implies the inequality  (\ref{eq:func_ineq2}).
\end{proof}
\section{Eventual smoothness of solutions to a chemotaxis-Navier--Stokes system with rotational flux components}
\label{section:application}
Having proven some essential and powerful tools to deal with the lack of easily accessible a priori information (due to the non-scalar sensitivity $S$ and full Navier--Stokes fluid model), we will now begin the journey toward the second result of this paper, namely the derivation of eventual smoothness properties for the solutions constructed in \Cref{theorem:ext_existence}.
\subsection{Generalized solution concept and approximate solutions}
As our first step, let us now briefly recall some key points from \cite{MyExistence} concerning these solutions. First, we want to clarify what is actually meant when talking about generalized mass-preserving solutions in \Cref{theorem:ext_existence}:
\begin{definition}
	\label{definition:weak_solution}
	Let $\Omega \subseteq \R^2$ be a bounded, convex domain with a smooth boundary and $f$, $S$ and $\phi$ be parameter functions that conform to (\ref{f_regularity})--(\ref{phi_regularity}). Further let $(n_0, c_0, u_0)$ be some initial data with the properties outlined in (\ref{initial_data_props}).
	\\[0.5em]
	We then call a triple of functions 
	\begin{align*}
	&n \in L^\infty((0,\infty); L^1(\Omega)), \\
	&c \in L^\infty_\loc(\overline{\Omega}\times[0,\infty)) \cap L^2_\loc([0,\infty); W^{1,2}(\Omega)) \;\; \text{ and } \;\; \numberthis\label{wsol:regularity} \\
	&u \in L^\infty_\loc([0,\infty);(L^2(\Omega))^2) \cap L^2_\loc([0,\infty);(W_0^{1,2
	}(\Omega))^2)		
	\end{align*}
	with $n \geq 0$, $c \geq 0, \div u = 0$ a.e.\ in $\Omega\times(0,\infty)$,
	\begin{equation}
	\int_{\Omega} n(\cdot,t) = \int_{\Omega} n_0 \;\;\;\; \text{ for a.e.\ } t > 0 \label{wsol:mass_perservation}
	\end{equation}
	and 
	\begin{equation}
	\ln(n + 1) \in L^2_\loc([0,\infty);W^{1,2}(\Omega))
	\label{wsol:ln_n_regularity}
	\end{equation}
	a global mass-preserving generalized solution of (\ref{problem})--(\ref{initial_data}) if the inequality
	\begin{align*}
	-\int_0^\infty \int_{\Omega} \ln(n+1)\varphi_t - \int_{\Omega} \ln(n_0 + 1)\varphi(\cdot, 0) \geq& \int_0^\infty \int_{\Omega} \ln(n+1)\laplace\varphi + \int_0^\infty \int_{\Omega} |\grad \ln(n+1)|^2\varphi \\
	& - \int_0^\infty \int_{\Omega} \frac{n}{n+1} \grad \ln(n+1) \cdot (S(x,n,c)\grad c)\varphi \\
	& + \int_0^\infty \int_{\Omega} \frac{n}{n+1} (S(x,n,c) \grad c)\cdot \grad \varphi \\
	& + \int_0^\infty \int_\Omega \ln(n+1)(u\cdot \grad \varphi) \numberthis \label{wsol:ln_n_inequality}
	\end{align*}
	holds for all nonnegative $\varphi\in C_0^\infty(\overline{\Omega}\times[0,\infty))$ with $\grad \varphi \cdot \nu = 0$ on $\partial \Omega\times[0,\infty)$, if further
	\begin{equation}
	\int_0^\infty\int_{\Omega} c\varphi_t + \int_{\Omega} c_0\varphi(0, \cdot) = \int_0^\infty \int_\Omega \grad c \cdot \grad \varphi + \int_0^\infty \int_\Omega n f(c) \varphi - \int_0^\infty \int_\Omega c(u \cdot \grad \varphi) \label{wsol:c_equality}
	\end{equation}
	holds for all $\varphi \in L^\infty(\Omega\times(0,\infty)) \cap L^2((0,\infty); W^{1,2}(\Omega))$ having compact support in $\overline{\Omega}\times[0,\infty)$ with $\varphi_t \in L^2(\Omega\times(0,\infty))$, and if finally 
	\begin{equation}
	-\int_0^\infty \int_\Omega u \cdot \varphi_t - \int_\Omega u_0 \cdot \varphi(\cdot, 0) = -\int_0^\infty \int_\Omega \grad u \cdot \grad \varphi + \int_0^\infty \int_\Omega (u \otimes u)\cdot\grad \varphi + \int_0^\infty\int_\Omega n\grad\phi \cdot \varphi
	\label{wsol:u_equality}
	\end{equation}
	holds for all $\varphi \in C_0^\infty(\overline{\Omega}\times[0,\infty); \R^2)$ with $\div \varphi = 0$ on $\Omega\times[0,\infty)$.
\end{definition} 
\noindent Second, let us review one important detail about the construction of the generalized solutions $(n, c, u)$ in \cite{MyExistence},
namely the fact that they are the (almost everywhere) pointwise limits of some approximate solutions \[((n_\eps, c_\eps, u_\eps, P_\eps))_{\eps \in (0,1)}\] with
\begin{align*}
	n_\eps, c_\eps &\in C^0(\overline{\Omega}\times[0,\infty))\cap C^{2,1}(\overline{\Omega}\times(0,\infty)), \\
	u_\eps &\in C^0(\overline{\Omega}\times[0,\infty);\R^2)\cap C^{2,1}(\overline{\Omega}\times(0,\infty);\R^2) \\
	P_\eps &\in C^{1,0}(\overline{\Omega}\times(0,\infty))
\end{align*}
and $c_\eps \geq 0, n_\eps > 0$
along a suitable sequence $(\eps_j)_{j\in\N} \subseteq (0,1)$ with $\eps_j \searrow 0$ as $j \rightarrow \infty$. The approximate solutions $(n_\eps, c_\eps, u_\eps, P_\eps)$ solve the following regularized version of (\ref{problem}) with (\ref{boundary_conditions}) and (\ref{initial_data}):
\[
	\left\{
	\begin{aligned}
	{n_\eps}_t + u_\eps \cdot \grad n_\eps \;=&\;\; \laplace n_\eps - \div (n_\eps S_\eps(x, n_\eps, c_\eps)\grad c_\eps), \;\;\;\;\;\;\; && x \in \Omega, \; t > 0,\\
	{c_\eps}_t + u_\eps \cdot \grad c_\eps \;=&\;\; \laplace c_\eps - n_\eps f(c_\eps), && x \in \Omega, \; t > 0,\\
	{u_\eps}_t + (u_\eps \cdot \grad ) u_\eps \;=&\;\; \laplace u_\eps + \grad P_\eps + n_\eps\grad \phi,  && x \in \Omega, \; t > 0, \\
	\div u_\eps \;=&\;\; 0,  && x \in \Omega, \; t > 0, \\
	\grad n_\eps \cdot \nu = \grad c_\eps \cdot \nu = 0&, \;\;\;\; u_\eps = 0, && x \in \partial\Omega, \; t > 0, \\
	n_\eps(x,0) = n_0(x), \;\; &c_\eps(x,0) = c_0(x), \;\; u_\eps(x,0) = u_0(x), && x \in \Omega
	\end{aligned}
	\numberthis \label{approx_system}
	\right..
\] 
The key difference of the above system in comparison to the original is that the sensitivity $S$ is approximated by functions $S_\eps$ while the rest of the system mostly stays the same. The $S_\eps$ are constructed in such a way that they are zero for large $n_\eps$, which makes global existence arguments for the approximate solutions straightforward, and in such a way that they vanish near the boundary, which simplifies the more complex no-flux boundary conditions to standard Neumann boundary conditions. They are further made to retain property (\ref{S_0_bound}) of $S$ with the same $S_0$ and converge pointwise to $S$ on $\Omega\times[0,\infty)\times[0,\infty)$ as $\eps \searrow 0$. For the full details of the construction, see \cite[Section 2]{MyExistence}.
\\[0.5em]
Having now established the necessary context, let us fix a few things for the remainder of this paper. The set $\Omega \subseteq \R^2$ is always a bounded, convex domain with a smooth boundary and $f$, $S$ and $\phi$ are always the parameter functions mentioned in (\ref{problem}) and are assumed to conform to (\ref{f_regularity})--(\ref{phi_regularity}). We further fix some initial data $(n_0, c_0, u_0)$ with the properties outlined in (\ref{initial_data_props}). Given all this, we can then fix a corresponding solution $(n,c,u)$ as constructed in \Cref{theorem:ext_existence} and the family $((n_\eps, c_\eps, u_\eps, P_\eps))_{\eps \in (0,1)}$ of approximate solutions and sequence $(\eps_j)_{j\in\N}$ used in said construction.

\subsection{Some improved initial observations adapted from the existence theory}
As our first step in analyzing the approximate solutions fixed above, we revisit some of their properties from \cite{MyExistence}, albeit after some slight modifications and with sometimes significant improvements. The first such properties are laid out in the following lemma, which is taken almost verbatim from \cite[Lemma 3.1]{MyExistence} and provides us with some initial, important, though sometimes rather weak, a priori information about the families $(n_\eps)_{\eps \in (0,1)}$ and $(c_\eps)_{\eps \in (0,1)}$:
\begin{lemma}
	\label{lemma:basic_props}
	The mass conservation equality
	\begin{equation}
	\int_\Omega n_\eps(\cdot,t)   = \int_\Omega n_0 \label{eq:mass_perservation}
	\end{equation}
	holds for all $t > 0$, $\eps \in (0,1)$ and, for each $p \in [1,\infty]$, the inequality
	\begin{equation}
	\|c_\eps(\cdot, t)\|_\L{p} \leq \|c_\eps(\cdot, s)\|_\L{p}
	\label{eq:c_monoticity}
	\end{equation}
	holds for all $t\geq s\geq 0$ and $\eps \in (0,1)$. We further have that 
	\begin{equation}
	\int_t^\infty\int_\Omega |\grad c_\eps|^2 \leq \frac{1}{2}\int_\Omega c_\eps^2(\cdot, t)
	\label{eq:grad_c_bound}
	\end{equation}
	for all $t > 0$, $\eps \in (0,1)$ and there exists $C > 0$ such that
	\begin{equation}
	\int_0^\infty\int_\Omega \frac{|\grad n_\eps|^2}{n_\eps^2} \leq C \label{eq:weak_grad_n_smallness}
	\end{equation}
	for all $\eps \in (0,1)$.
\end{lemma}
\begin{proof}
	Similar to the methods seen in \cite[Lemma 3.1]{MyExistence} and \cite[Lemma 2.3]{WinklerStokesCase}, these inequalities follow immediately from testing the first equation in (\ref{approx_system}) with $1$ and $1/n_\eps$ and from testing the second equation in (\ref{approx_system}) with $c_\eps^p$ for $p \in [1,\infty)$. The case $p = \infty$ in (\ref{eq:c_monoticity}) then follows by taking the limit $p \rightarrow \infty$.
\end{proof}
\noindent As (in a sense) weaker versions of the functional inequalities (\ref{eq:func_ineq1}) and (\ref{eq:func_ineq2}) already played a significant role in deriving some of the a priori estimates used for the existence theory in \cite{MyExistence}, we will now use our improved inequalities to derive similar but stronger versions of some of the integrability properties already used in said existence theory. Namely, we manage to extend some local time integrability properties used in \cite{MyExistence} to global time integrability properties due to the elimination of a problematic additive mass term in both of the functional inequalities from \cite[Lemma 3.2]{MyExistence} at the cost of $\beta_0$ becoming potentially very small. This makes the results much more useful for long time behavior considerations as these types of integrability properties are in a sense already a weak indication for the stabilization of our solutions as $t \rightarrow \infty$.
\\[0.5em]
Our first target for this will be a straightforward improvement of \cite[Lemma 3.3]{MyExistence} by using (\ref{eq:func_ineq2}):
\begin{corollary}	\label{lemma:nlnn_integrabilty}
	There exists $C > 0$ such that
	\[
		\int_0^\infty\int_\Omega n_\eps \ln\left( \frac{n_\eps}{\;\overline{n_0}\;} \right) \leq C
	\]
	for all $\eps \in (0,1)$ with $\overline{n_0} \defs \frac{1}{|\Omega|}\int_\Omega n_0$.
\end{corollary}
\begin{proof}
	Combining the inequalities (\ref{eq:mass_perservation}) and (\ref{eq:weak_grad_n_smallness}) from \Cref{lemma:basic_props} with the functional inequality (\ref{eq:func_ineq2}) from \Cref{theorem:func_ineq} directly yields this.
\end{proof}

\noindent Again by using our key new functional inequalities from \Cref{theorem:func_ineq}, we can now improve the argument used in \cite[Lemma 3.4]{MyExistence} to gain the following lemma:
\begin{lemma}
	There exists $C > 0$ such that
	\begin{equation*}
	\int_0^\infty \int_{\Omega} |u_\eps|^2 \leq C,  \;\;  \int_0^\infty \int_{\Omega} |\grad u_\eps|^2 \leq C
	\end{equation*}
	for all $\eps \in (0,1)$.\label{lemma:u_props}
\end{lemma}
\begin{proof}
	We first test the third equation in (\ref{approx_system}) with $u_\eps$ to gain that
	\begin{align*}
	\frac{1}{2} \frac{\d}{\d t}\int_{\Omega} |u_\eps|^2 =& -\int_{\Omega} |\grad u_\eps|^2 + \int_\Omega n_\eps \grad \phi \cdot u_\eps \\
	=& -\int_{\Omega} |\grad u_\eps|^2 + \int_\Omega (n_\eps - \overline{n_0}) (\grad \phi \cdot u_\eps)
	\end{align*}
	with $\overline{n_0} \defs \frac{1}{|\Omega|}\int_\Omega n_0$, which we can then improve via our functional inequality (\ref{eq:func_ineq1}) from \Cref{theorem:func_ineq} to
	\begin{align*}
	\frac{1}{2} \frac{\d}{\d t}\int_{\Omega} |u_\eps|^2
	\leq& -\int_{\Omega} |\grad u_\eps|^2 + \frac{1}{a}\int_{\Omega} n_\eps \ln\left( \frac{n_\eps}{ \overline{n_0} }\right) + \frac{a}{4\beta_0} \left\{ \int_\Omega n_0 \right\}\int_{\Omega}|\grad (\grad \phi \cdot u_\eps)|^2 \numberthis \label{eq:u_weak_1}
	\end{align*}
	for any $a > 0$ and all $t > 0$ as well as $\eps \in (0,1)$. We now further note that
	\begin{align*}
	\int_{\Omega}|\grad (\grad \phi \cdot u_\eps)|^2  &\leq  2\int_{\Omega} |\grad \phi|^2 |\grad u_\eps|^2 + 2\int_{\Omega} |D^2 \phi|^2 |u_\eps|^2 \\
	&\leq 2\|\grad \phi\|_\L{\infty}^2\int_{\Omega} |\grad u_\eps|^2 + 2\|D^2 \phi\|_\L{\infty}^2 \int_{\Omega} |u_\eps|^2 
	\leq K_1 \int_{\Omega}|\grad u_\eps|^2 
	\end{align*}
	with $K_1 \defs 2\|\grad \phi\|_\L{\infty}^2 + 2\|D^2 \phi\|_\L{\infty}^2 C_\text{p}^2$ for all $t > 0$ and $\eps \in (0,1)$ due to the Poincaré inequality. Here, $D^2  \phi$ is the Hessian of $\phi$ and $C_\text{p}$ is the Poincaré constant for $\Omega$. If we now choose \[
	a \defs \frac{2\beta_0}{K_1} \left\{ \int_\Omega n_0 \right\}^{-1}
	\] in the inequality (\ref{eq:u_weak_1}), we gain that
	\[
	\frac{1}{2} \frac{\d}{\d t}\int_{\Omega}  |u_\eps|^2 \leq -\frac{1}{2}\int_{\Omega} |\grad u_\eps|^2 + \frac{K_1}{2\beta_0}\left\{ \int_\Omega n_0 \right\} \int_{\Omega} n_\eps \ln\left( \frac{n_\eps}{ \overline{n_0} }\right) \;\;\;\; \text{ for all } t > 0 \text{ and } \eps \in (0,1). 
	\]
	After time integration and some rearranging, we then further see that
	\[
	\int_0^\infty \int_{\Omega} |\grad u_\eps|^2  \leq \int_{\Omega}  |u_0|^2
	+ \frac{K_1}{\beta_0}\left\{ \int_\Omega n_0 \right\} \int_0^\infty \int_{\Omega} n_\eps \ln\left( \frac{n_\eps}{ \overline{n_0} }\right) \;\;\;\; \text{ for all } \eps \in (0,1).
	\]
	By the integrability property laid out in  \Cref{lemma:nlnn_integrabilty}, there then moreover exists a constant $K_2 > 0$ such that
	\[
	\int_0^\infty \int_{\Omega} |\grad u_\eps|^2 \leq  \int_{\Omega} |u_0|^2 + \frac{K_1K_2}{\beta_0}\left\{ \int_\Omega n_0 \right\} \;\;\;\; \text{ for all } \eps \in (0,1).
	\]
	This together with one last application of the Poincaré inequality completes the proof.
\end{proof}

\subsection{Eventual smallness of the family $(c_\eps)_{\eps \in (0,1)}$ in $L^p(\Omega)$ for all $p \in [1,\infty)$}

As the second equation in (\ref{approx_system}) is in many ways the easiest to handle, it is not surprising that the first fairly strong result of this section is in fact concerned with the family $(c_\eps)_{\eps \in (0,1)}$. Namely, we will now prove that the $L^p(\Omega)$-norms of said family are not only monotonically decreasing for $p\in[1,\infty)$ as seen in \Cref{lemma:basic_props}, but that they in fact tend to zero for $t \rightarrow \infty$ in an $\eps$-independent fashion. Similar to our prior results in \Cref{lemma:nlnn_integrabilty} and \Cref{lemma:u_props}, this is again heavily based on our new functional inequalities from \Cref{theorem:func_ineq}.  

\begin{lemma}
	\label{lemma:c_stabilization}
	For each $p\in [1,\infty)$ and $\delta > 0$, there exists a time $t_0 = t_0(\delta, p) > 0$ such that
	\[
		\|c_\eps(\cdot, t)\|_\L{p} \leq \delta
	\] 
	for all $t \geq t_0$ and $\eps \in (0,1)$.
\end{lemma}
\begin{proof}	
	As our first step, we integrate the second equation in (\ref{approx_system}) to gain that
	\[
		\frac{\d}{\d t} \int_\Omega c_\eps = -\int_\Omega n_\eps f(c_\eps) \;\;\;\; \text{ for all } t > 0 \text{ and } \eps \in (0,1),
	\]	
	which implies that
	\begin{equation}
		\int_0^\infty \int_\Omega n_\eps f(c_\eps) \leq \int_\Omega c_0 \;\;\;\;\;\; \text{ for all } \eps \in (0,1)	
		\label{eq:nfc_integrability}
	\end{equation}
	by time integration.
	We now rewrite $\int_\Omega f(c_\eps)$ as follows:
	\[
	   \int_\Omega f(c_\eps) = \frac{1}{\;\overline{n_0}\;} \left[
	   \int_\Omega (\overline{n_0} - n_\eps) f(c_\eps) + \int_\Omega n_\eps f(c_\eps) \right]	\;\;\;\; \text{ for all } t > 0 \text{ and } \eps \in (0,1) \text{ with } \overline{n_0} \defs \frac{1}{|\Omega|}\int_\Omega n_0.				
	\]
	This then allows us to apply functional inequality (\ref{eq:func_ineq1}) from \Cref{theorem:func_ineq} (setting $a = 2$) to further see that
	\begin{align*}
		\int_\Omega f(c_\eps) &\leq \frac{1}{\;\overline{n_0}\;} \left[ \frac{1}{2} \int_\Omega n_\eps \ln\left( \frac{n_\eps}{\;\overline{n_0}\;} \right) + \frac{1}{2\beta_0} \left\{ \int_\Omega n_0 \right\} \int_\Omega |\grad f(c_\eps)|^2 + \int_\Omega n_\eps f(c_\eps) \right] \\
		&\leq \frac{1}{\;\overline{n_0}\;} \left[ \frac{1}{2} \int_\Omega n_\eps \ln\left( \frac{n_\eps}{\;\overline{n_0}\;} \right) + \frac{K_1^2}{2\beta_0} \left\{ \int_\Omega n_0 \right\} \int_\Omega |\grad c_\eps|^2 + \int_\Omega n_\eps f(c_\eps) \right]	\;\;\;\; \text{ for all } \eps \in (0,1)	
	\end{align*}
	with $\beta_0$ as in \Cref{theorem:func_ineq} and $K_1 \defs \|f'\|_{L^\infty([0, \|c_0\|_\L{\infty}])}$. Considering the integrability properties in \Cref{lemma:nlnn_integrabilty}, \Cref{lemma:basic_props} and (\ref{eq:nfc_integrability}), there must therefore exist a constant $K_2 > 0$ such that
	\begin{equation}
		\int_0^\infty \int_\Omega f(c_\eps) \leq K_2 \;\;\;\; \text{ for all } \eps \in (0,1).
		\label{eq:fc_integrabilty}
	\end{equation}
	We now fix $\delta > 0$ and then let $\xi \defs \frac{\delta}{2|\Omega|}$. Because $f$ is positive outside of zero and continuous, there must exist a constant $K_3 > 0$ such that
	\[
		f(y) \geq K_3 \;\;\;\; \text{ for all } y \in [\xi, \|c_0\|_\L{\infty}].
	\]
	Because (\ref{eq:fc_integrabilty}) implies that
	\[
		\frac{1}{t_0}\int_0^{t_0} \int_\Omega f(c_\eps) \leq  \frac{K_3}{\|c_0\|_\L{\infty}}\frac{\delta}{2} \;\;\;\; \text{ for all } \eps \in (0,1)
	\] with \[
		t_0 \defs \frac{K_2 \|c_0\|_\L{\infty}}{K_3} \frac{2}{\delta} > 0,
	\]
	we can, for each $\eps \in (0,1)$, find at least one $t_\eps \in (0,t_0)$ such that
	\[
		\int_\Omega f(c_\eps(\cdot, t_\eps)) \leq \frac{K_3}{\|c_0\|_\L{\infty}}\frac{\delta}{2}.
	\]
	Using this, we then gain that
	\begin{align*}
		\int_\Omega c_\eps(\cdot, t_\eps) &= \int_{\{ c_\eps(\cdot, t_\eps) \leq \xi \}} c_\eps(\cdot, t_\eps) + \int_{\{ c_\eps(\cdot, t_\eps) > \xi \}} c_\eps(\cdot, t_\eps) \\
		&\leq |\Omega|\xi + \frac{\|c_0\|_\L{\infty}}{K_3} \int_\Omega f(c_\eps(\cdot, t_\eps)) \\
		&\leq \frac{\delta}{2} + \frac{\delta}{2} = \delta \;\;\;\; \text{ for all } \eps \in (0,1)
	\end{align*}
	and therefore 
	\[
		\int_\Omega c_\eps(\cdot, t) \leq \delta \;\;\;\; \text{ for all } \eps \in (0,1) \text{ and } t \geq t_0
	\]
	because of the monotonicity properties for the family $(c_\eps)_{\eps \in (0,1)}$ seen in \Cref{lemma:basic_props} and the fact that $t_\eps < t_0$ for all $\eps \in (0,1)$.
	This is exactly our desired result for $p = 1$ and, because \Cref{lemma:basic_props} further gives us a global uniform $L^\infty$ bound for the family $(c_\eps)_{\eps \in (0,1)}$, our desired result follows for $p > 1$ by interpolation.
\end{proof} \noindent
By combining the above lemma with (\ref{eq:grad_c_bound}) from \Cref{lemma:basic_props}, we then immediately gain an  important corollary about the gradients of the family $(c_\eps)_{\eps \in (0,1)}$.
\begin{corollary}
	For each $\delta > 0$, there exists a time $t_0 = t_0(\delta) > 0$ such that
	\[
		\int_{t_0}^\infty \int_\Omega |\grad c_\eps|^2 \leq \delta
	\]
	for all $\eps \in (0,1)$.
	\label{lemma:grad_c_long_time_small}
\end{corollary}  
\begin{proof}
	Because of \Cref{lemma:c_stabilization}, there exists $t_0 > 0$ such that 
	\[
		\int_\Omega c_\eps^2(\cdot, t_0) \leq 2\delta \;\;\;\; \text{ for all } \eps \in (0,1).
	\]
	The inequality (\ref{eq:grad_c_bound}) from \Cref{lemma:basic_props} then immediately implies our desired result.
\end{proof}

\subsection{Eventual smallness of a key functional and its associated norms}
\label{section:application_final_estimates}
Having now leveraged the functional inequalities in \Cref{theorem:func_ineq} to overcome some critical gaps in a priori information, our next step will be to further improve upon our (sometimes fairly weak) stabilization results in the previous sections. 
\\[0.5em]
As a first step toward this goal, we will show that, if the functional $\mathcal{F}_\eps$ seen in (\ref{eq:f_eps_functional}) is small at some time $t_0 > 0$, it in fact stays at least somewhat small from there on out. While the functional itself is already composed of some key integrals, 
the argument for this also gives us that time-space integrals over some higher-order derivatives of our solution components become small when considered only from $t_0$ onward. This approach is inspired by similar methods seen in \cite{WinklerSmallMass}.
\begin{lemma}\label{lemma:small_things_stay_small}
	There exist constants $\delta_0 > 0$, $C \geq 1$ such that the following holds for all $\delta \in (0,\delta_0)$: \\ 
	Let $\eps \in (0,1)$. If there exists $t_0 > 0$ such that the functional
	\begin{equation}
	\mathcal{F}_\eps(t) \defs \int_{\Omega} n_\eps \ln\left( \frac{n_\eps}{\overline{n_0}}\right) + \frac{1}{2}\int_{\Omega} |\grad c_\eps|^2 + \frac{1}{2C }\int_{\Omega} |u_\eps|^2 \;\;\;\;\text{ for all } t > 0, \text{ where }\overline{n_0} \defs \frac{1}{|\Omega|}\int_\Omega n_0, \label{eq:f_eps_functional}
	\end{equation}
	has the property
	\begin{equation}
		\mathcal{F_\eps}(t_0) \leq \frac{\delta}{8C} \label{eq:initially_small}
	\end{equation}
	and further the inequality 
	\begin{equation}
		\int_{t_0}^\infty\int_{\Omega}|\grad c_\eps|^2 \leq  \frac{\delta}{8C^2} \label{eq:grad_c_int_small}
	\end{equation}
	holds, then
	\begin{equation}
		\mathcal{F_\eps}(t) \leq \delta  \label{eq:things_stay_small_res_1}
	\end{equation}
	for all $t \geq t_0$ and
	\begin{equation}
		\int_{t_0}^\infty \int_{\Omega} \frac{|\grad n_\eps|^2}{n_\eps} \leq \delta, \;\;\;\;  \int_{t_0}^\infty \int_{\Omega} |\laplace c_\eps|^2 \leq \delta, \;\;\;\;  \int_{t_0}^\infty \int_{\Omega} |\grad u_\eps|^2 \leq \delta. \label{eq:things_stay_small_res_2}
	\end{equation}
\end{lemma}
\begin{proof}
	Before we present the actual core of this proof, let us first fix some necessary constants to streamline later arguments and make sure that there are no hidden interdependencies: 
	\\[0.5em]
	Let first $C_\text{p} > 0$ be the constant used in the following well-known Poincaré inequalities on $\Omega$ (cf.\ \cite[p.~290]{BrezisFAandPDE}):
	\begin{equation}
		\|\varphi\|_\L{2} \leq C_\text{p} \|\grad \varphi\|_\L{2} \;\;\;\; \text{ for all } \varphi \in C^1(\overline{\Omega}) \text{ with } \varphi = 0 \text{ on }\partial\Omega \label{eq:poincare_1}
	\end{equation}
	and
	\begin{equation}
		\|\grad \varphi\|_\L{2} \leq C_\text{p} \|\laplace \varphi\|_\L{2} \;\;\;\; \text{ for all } \varphi \in C^2(\overline{\Omega}) \text{ with } \grad \varphi \cdot \nu = 0 \text{ on } \partial \Omega
		\label{eq:poincare_2}
	\end{equation}
	Let now $C_\text{s} > 0$ be the constant in the following Sobolev-type inequality (cf.\ \cite[p.~313]{BrezisFAandPDE}):
	\begin{equation} \label{eq:sobolev}
		\int_{\Omega} |\varphi - \overline{\varphi}|^2 \leq C_\text{s} \left\{ \int_{\Omega} |\grad \varphi| \right\}^2 \;\;\;\; \text{ for all } \varphi \in C^1(\overline{\Omega}) \text{ with } \overline{\varphi} \defs \frac{1}{|\Omega|} \int_\Omega \varphi
	\end{equation}
	Finally, let $C_\text{gni} > 0$ be such that the inequality
	\begin{equation} \label{eq:gni}
		\int_\Omega |\grad  \varphi|^4 \leq C_\text{gni} \left\{ \int_{\Omega} |\grad \varphi|^2\right\} \int_{\Omega} |\laplace  \varphi|^2  \;\;\;\; \text{ for all } \varphi \in C^2(\overline{\Omega}) \text{ with } \grad \varphi \cdot \nu = 0 \text{ on } \partial\Omega,
	\end{equation}
	which can be derived from the Gagliardo--Nirenberg inequality (cf.\ \cite{GNIClassic}), elliptic regularity theory (cf.\ \cite[Theorem 19.1]{FriedmanPDE}), the Poincaré inequality (cf.\ \cite[p.~312]{BrezisFAandPDE}) and (\ref{eq:poincare_2}), holds. 
	\\[0.5em]
	Let us further fix the basic constants
	\[
		K_1 \defs S_0(\|c_0\|_\L{\infty}),\;\;\;\; K_2 \defs  \|f\|_{C^1([0, \|c_0\|_\L{\infty}])},\;\;\;\; K_3 \defs \|\grad \phi\|_\L{\infty}, \;\;\;\; K_4 \defs 2K_1^2 + 2K_2^2 + K_2,
	\]
	which only depend on the parameters of the system (\ref{problem}) and the initial data $c_0$. Then let
	\[
	C \defs \max\left(2K_3^2 C^2_\text{p} C_\text{s} m_0, K_4 \frac{m_0}{|\Omega|},  1 \right), \;\;\;\; K_5 \defs C_\text{s} m_0 K_4^2 + C  \;\;\;\;\text{ and }\;\;\;\;	\delta_0 \defs \frac{1}{4K_5 C_\text{gni}}
	\]
	with $m_0 \defs \int_\Omega n_0$.
	For the actual core of this proof, we now fix $\eps \in (0,1)$, $\delta \in (0,\delta_0)$ and $t_0 > 0$ such that (\ref{eq:initially_small}) and (\ref{eq:grad_c_int_small}) hold with $C$ and $\delta_0$ as defined above. We then test each of the first three equations in (\ref{approx_system}) with certain appropriate test functions:
	\\[0.5em]
	We test the first equation with $\ln(n_\eps)$ and use Young's inequality to see that
	\begin{align*}
	\frac{\d}{\d t}\int_{\Omega} n_\eps\ln\left(\frac{n_\eps}{\overline{n_0}}\right)
	=& - \int_{\Omega} \frac{|\grad n_\eps|^2}{n_\eps} + \int_{\Omega} \grad n_\eps \cdot S_\eps(x,n_\eps,c_\eps) \grad c_\eps \\
	\leq& -\frac{7}{8}\int_{\Omega} \frac{|\grad n_\eps|^2}{n_\eps} + 2K_1^2 \int_{\Omega} |\grad c_\eps|^2 n_\eps \numberthis \label{eq:things_stay_small_1}
	\end{align*}
	for all $t > 0$.
	We then test the second equation with $-\laplace c_\eps$ to see that
	\begin{align*}
	&\frac{1}{2}\frac{\d}{\d t}\int_{\Omega} |\grad c_\eps|^2 \\
	=&  - \int_{\Omega} |\laplace c_\eps|^2 + \int_\Omega (u_\eps \cdot \grad c_\eps) \laplace c_\eps + \int_{\Omega}n_\eps f(c_\eps) \laplace c_\eps \\
	=& - \int_{\Omega} |\laplace c_\eps|^2 - \int_\Omega (\grad u_\eps  \grad c_\eps) \cdot \grad c_\eps - \frac{1}{2}\int_\Omega u_\eps \cdot \grad |\grad c_\eps|^2 - \int_{\Omega}f(c_\eps) \grad n_\eps\cdot \grad c_\eps - \int_{\Omega}n_\eps f'(c_\eps) |\grad c_\eps|^2  \\
	\leq& - \int_{\Omega} |\laplace c_\eps|^2 + \frac{1}{4C} \int_{\Omega} |\grad u_\eps|^2 + C \int_{\Omega} |\grad c_\eps|^4 + \frac{1}{8}  
	\int_{\Omega} \frac{|\grad n_\eps|^2}{n_\eps} + (2K_2^2 + K_2)  \int_{\Omega} |\grad c_\eps|^2 n_\eps \numberthis \label{eq:things_stay_small_2}
	\end{align*}
	for all $t > 0$ by again using Young's inequality and the fact that
	\[
		\int_\Omega u_\eps \cdot \grad |\grad c_\eps|^2 = -\int_\Omega (\div u_\eps)  |\grad c_\eps|^2 = 0 
	\] because $\div u_\eps \equiv 0$. Similar to the proof of \Cref{lemma:u_props}, we test the third equation with $u_\eps$ to see that
	\begin{align*}
	\frac{1}{2}\frac{\d}{\d t}  \int_\Omega |u_\eps|^2 
	=& -\int_{\Omega} |\grad u_\eps|^2 + \int_{\Omega} (n_\eps - \overline{n_0}) (\grad \phi \cdot u_\eps) \\
	\leq& -\int_{\Omega} |\grad u_\eps|^2 +  \frac{K_3^2 C^2_\text{p}}{2} \int_{\Omega} (n_\eps - \overline{n_0})^2 +  \frac{1}{2 C^2_\text{p}} \int_\Omega |u_\eps|^2 \\
	\leq& -\frac{1}{2}\int_{\Omega} |\grad u_\eps|^2 + \frac{K_3^2 C^2_\text{p} C_\text{s}}{2} \left\{ \int_{\Omega} |\grad n_\eps| \right\}^2  \\
	\leq& -\frac{1}{2}\int_{\Omega} |\grad u_\eps|^2 + \frac{K_3^2 C^2_\text{p} C_\text{s}}{2} \left\{ \int_\Omega n_\eps \right\}  \int_{\Omega} \frac{|\grad n_\eps|^2}{n_\eps}  \\
	\leq& -\frac{1}{2}\int_{\Omega} |\grad u_\eps|^2 + \frac{C}{4} \int_{\Omega} \frac{|\grad n_\eps|^2}{n_\eps} \numberthis \label{eq:things_stay_small_3}
	\end{align*}
	for all $t > 0$ by using Young's inequality, the Hölder inequality, (\ref{eq:poincare_1}) and (\ref{eq:sobolev}).
	\\[0.5em]
	Inequality (\ref{eq:things_stay_small_1}) and (\ref{eq:things_stay_small_2}) then combine to give us
	\begin{align*}
	&\frac{\d}{\d t}\int_{\Omega} n_\eps\ln\left(\frac{n_\eps}{\overline{n_0}}\right) + \frac{1}{2}\frac{\d}{\d t}\int_{\Omega} |\grad c_\eps|^2  \\
	\leq& - \frac{3}{4} \int_{\Omega} \frac{|\grad n_\eps|^2}{n_\eps} -\int_{\Omega} |\laplace c_\eps|^2 + \frac{1}{4C} \int_{\Omega} |\grad u_\eps|^2 + C \int_{\Omega} |\grad c_\eps|^4 + K_4 \int_{\Omega} |\grad c_\eps|^2 n_\eps
	\end{align*}
	for all $t > 0$. The most critical term here is $\int_{\Omega} |\grad c_\eps|^2 n_\eps$, which we therefore further estimate in a similar fashion to the arguments seen in (\ref{eq:things_stay_small_3}) as
	\begin{align*}
	\int_{\Omega} |\grad c_\eps|^2 n_\eps &= 
	\int_{\Omega} |\grad c_\eps|^2 (n_\eps - \overline{n_0}) + \frac{m_0}{|\Omega|} \int_{\Omega} |\grad c_\eps|^2 \\
	&\leq C_\text{s} m_0 K_4 \int_\Omega |\grad c_\eps|^4 + \frac{1}{4 C_\text{s} m_0  K_4 } \int_\Omega (n_\eps - \overline{n_0})^2 + \frac{m_0}{|\Omega|} \int_{\Omega} |\grad c_\eps|^2 \\
	&\leq C_\text{s} m_0 K_4 \int_{\Omega} |\grad c_\eps|^4 + \frac{1}{4K_4} \int_{\Omega} \frac{|\grad n_\eps|^2}{n_\eps} + \frac{m_0}{|\Omega|} \int_{\Omega} |\grad c_\eps|^2 \;\;\;\; \text{ for all } t > 0,
	\end{align*}
	to gain that 
	\begin{align*}
	&\frac{\d}{\d t}\int_{\Omega} n_\eps\ln\left(\frac{n_\eps}{\overline{n_0}}\right) + \frac{1}{2}\frac{\d}{\d t}\int_{\Omega} |\grad c_\eps|^2 \\  
	\leq& - \frac{1}{2} \int_{\Omega} \frac{|\grad n_\eps|^2}{n_\eps} -\int_{\Omega} |\laplace c_\eps|^2 + \frac{1}{4C} \int_{\Omega} |\grad u_\eps|^2 + K_5 \int_{\Omega} |\grad c_\eps|^4 + \mathcal{G}_\eps(t)
	\end{align*}
	with $\mathcal{G}_\eps(t) \defs K_4\frac{m_0}{|\Omega|}\int_{\Omega} |\grad c_\eps(\cdot, t)|^2$ for all $t > 0$.
	This combined with (\ref{eq:things_stay_small_3}) then finally gives us
	\begin{align*}
	&\frac{\d}{\d t}\int_{\Omega} n_\eps\ln\left(\frac{n_\eps}{\overline{n_0}}\right) + \frac{1}{2}\frac{\d}{\d t}\int_{\Omega} |\grad c_\eps|^2 +  \frac{1}{2C}\frac{\d}{\d t}\int_{\Omega} |u_\eps|^2   \\
	\leq& - \frac{1}{4} \int_{\Omega} \frac{|\grad n_\eps|^2}{n_\eps} -\int_{\Omega} |\laplace c_\eps|^2 - \frac{1}{4C} \int_{\Omega} |\grad u_\eps|^2 + K_5 \int_{\Omega} |\grad c_\eps|^4 +  \mathcal{G}_\eps(t),
	\end{align*}
	which can be further rewritten as
	\begin{align*}
	\mathcal{F}'_\eps(t) + \frac{1}{4} \int_\Omega \frac{|\grad n_\eps|^2}{n_\eps} + \int_{\Omega} |\laplace c_\eps|^2 + \frac{1}{4C} \int_{\Omega} |\grad u_\eps|^2 \leq K_5 \int_{\Omega} |\grad c_\eps|^4 + \mathcal{G}_\eps(t)
	\end{align*}
	for all $t > 0$.
	Because we further know that
	\begin{align*}
	\int_\Omega |\grad c_\eps|^4 \leq C_\text{gni} \left\{ \int_{\Omega} |\grad c_\eps|^2\right\} \int_{\Omega} |\laplace c_\eps|^2 \leq 2C_\text{gni} \mathcal{F}_\eps(t) \int_{\Omega} |\laplace c_\eps|^2
	\end{align*}
	due to (\ref{eq:gni}) and the fact that Jensen's inequality ensures that $\int_{\Omega} n_\eps\ln(\frac{n_\eps}{\overline{n_0}}) \geq 0$, we finally gain that
	\[
	\mathcal{F}'_\eps(t) + \frac{1}{4} \int_\Omega \frac{|\grad n_\eps|^2}{n_\eps} + \left(1 - 2 K_5 C_\text{gni} \mathcal{F}_\eps(t) \right) \int_{\Omega} |\laplace c_\eps|^2 + \frac{1}{4C} \int_{\Omega} |\grad u_\eps|^2 \leq \mathcal{G}_\eps(t)
	\]
	for all $t > 0$.
	We then let 
	\[
	S_\eps \defs \left\{ \, T > t_0 \;|\; \mathcal{F}_\eps(t) \leq \delta  \text{ for all } t\in(t_0, T) \, \right\},
	\]
	which is non-empty because of the continuity of $\mathcal{F}_\eps$, (\ref{eq:initially_small}) and the fact that $C \geq 1$. For all $T \in S_\eps$ and $t \in (t_0, T)$, we know that
	\[
		\mathcal{F}'_\eps(t) + \frac{1}{4} \int_\Omega \frac{|\grad n_\eps|^2}{n_\eps} + \frac{1}{2} \int_{\Omega} |\laplace c_\eps|^2 + \frac{1}{4C} \int_{\Omega} |\grad u_\eps|^2 \leq \mathcal{G}_\eps(t),
		\numberthis \label{eq:things_stay_small_4}
	\]
	because
	\[
		\left(1 - 2 K_5 C_\text{gni} \mathcal{F}_\eps(t) \right) \geq \left(1 - 2 K_5 C_\text{gni} \delta \right) \geq \left(1 - 2 K_5 C_\text{gni} \delta_0 \right) = \frac{1}{2}.
	\]
	If we now integrate in (\ref{eq:things_stay_small_4}) from $t_0$ to $t$, we gain that
	\begin{align*}
		&\;\mathcal{F}_\eps(t) + \frac{1}{4} \int_{t_0}^t  \int_\Omega \frac{|\grad n_\eps|^2}{n_\eps} + \frac{1}{2} \int_{t_0}^t  \int_{\Omega} |\laplace c_\eps|^2 + \frac{1}{4C} \int_{t_0}^t \int_{\Omega} |\grad u_\eps|^2 \\
		\leq&\; \mathcal{F}_\eps(t_0) + \int_{t_0}^\infty \mathcal{G}_\eps(s) \d s \leq \frac{\delta}{8C} + K_4\frac{m_0}{|\Omega|} \int_{t_0}^\infty \int_\Omega |\grad c_\eps|^2 \leq \frac{\delta}{8C} + \frac{\delta}{8C} \frac{m_0 K_4}{C |\Omega|} \leq \frac{ \delta}{4C} \leq \frac{\delta}{4} \numberthis \label{eq:small_things_stay_small_final}
	\end{align*}
	for all $T \in S_\eps$ and $t \in (t_0, T)$ because of (\ref{eq:initially_small}) and (\ref{eq:grad_c_int_small}) and the fact that $C \geq 1$ and $\frac{m_0 K_4}{|\Omega| C} \leq 1$ by definition of $C$. Therefore, $S_\eps = [t_0, \infty)$ due to the continuity of $S_\eps$ and thus (\ref{eq:small_things_stay_small_final}) holds on the entire interval $[t_0, \infty)$.
	\\[0.5em]
	This then directly implies (\ref{eq:things_stay_small_res_1}) and (\ref{eq:things_stay_small_res_2}) and consequently completes the proof.
\end{proof}
\noindent We can now use the above insight in combination with the properties derived in \Cref{lemma:nlnn_integrabilty}, \Cref{lemma:u_props} and \Cref{lemma:grad_c_long_time_small} to show that the smallness conditions for the functional $\mathcal{F}_\eps$ and gradient of $c_\eps$ are in fact achievable for every $\delta > 0$ at some $\eps$-independent time $t_0$ and that therefore the functional and thus its individual components become small in a uniform fashion. As an added bonus, we naturally also gain some uniform higher order smallness information for some of the dissipative terms, which will prove useful later on. 
\begin{lemma}
	\label{lemma:things_are_small}
	For each $\delta > 0$, there exists $t_0 = t_0(\delta) > 0$ such that
	\[
		\int_\Omega n_\eps(\cdot, t) \ln\left( \frac{n_\eps(\cdot, t)}{\;\overline{n_0}\;} \right)\leq \delta,\;\;  \int_\Omega |\grad c_\eps(\cdot, t)|^2 \leq \delta,\;\;  \int_\Omega |u_\eps(\cdot, t)|^2 \leq \delta
	\]
	for all $t > t_0$, $\eps \in (0,1)$ with $\overline{n_0} \defs \frac{1}{|\Omega|}\int_\Omega n_0$ and
	\[
		\int_{t_0}^\infty \int_{\Omega} \frac{|\grad n_\eps|^2}{n_\eps} \leq \delta, \;\; \int_{t_0}^\infty \int_{\Omega} |\laplace c_\eps|^2 \leq \delta, \;\; \int_{t_0}^\infty \int_{\Omega} |\grad u_\eps|^2 \leq \delta
	\]
	for all $\eps \in (0,1)$.
\end{lemma} 
\begin{proof}
	Let $\delta_0 > 0$, $C > 0$ and the functional $\mathcal{F}_\eps$ be as in \Cref{lemma:small_things_stay_small}. Without loss of generality, we assume $\delta < \delta_0$.
	Because of \Cref{lemma:grad_c_long_time_small}, we find $t_1 > 0$ such that
	\[
		\int_{t_1}^\infty \int_\Omega |\grad c_\eps|^2 \leq \frac{\delta}{8 C^2} \;\;\;\; \text{ for all } \eps \in (0,1).
	\]
	Because of
	\Cref{lemma:nlnn_integrabilty}, \Cref{lemma:basic_props}, and \Cref{lemma:u_props}, we further know that there exists a constant $K_1 > 0$ such that
	\[
		\frac{1}{t - t_1} \int_{t_1}^{t} \mathcal{F}_\eps(s) \d s \leq \frac{1}{t - t_1} \left[ \int_{0}^{\infty} \int_\Omega n_\eps \ln\left(\frac{n_\eps}{\;\overline{n_0}\;}\right) + \frac{1}{2} \int_{0}^\infty \int_\Omega |\grad c_\eps|^2 + \frac{1}{2C} \int_{0}^{\infty}\int_\Omega |u_\eps|^2  \right] \leq \frac{K_1}{t - t_1} 
	\]
	for all $\eps \in (0,1)$ and $t > t_1$. This implies that, for $t_0 \defs \frac{8 C K_1}{\delta} + t_1$ and all $\eps \in (0,1)$, we have
	\[
		\frac{1}{t_0 - t_1} \int_{t_1}^{t_0} \mathcal{F}_\eps(s) \d s \leq \frac{\delta}{8 C}.
	\]
	Therefore for each $\eps \in (0,1)$, there must exist $t_\eps \in (t_1, t_0)$ such that
	\[
		\mathcal{F}_\eps(t_\eps) \leq \frac{\delta}{8 C}.
	\]
	This directly implies our desired result from $t_\eps$ and therefore from $t_0$ onward for all $\eps \in (0,1)$ by application of \Cref{lemma:small_things_stay_small}. 
\end{proof}
\subsection{Eventual boundedness of $\|n_\eps\|_\L{\infty}$, $\|\grad c_\eps\|_\L{p}$, $\|A^\beta u_\eps\|_\L{p}$ for $p \in [1,\infty)$ and $\beta \in (\frac{1}{2}, 1)$ via bootstrap arguments}\label{section:bootstrap}
The eventual smallness and integrability results of the previous section now give us a critical foothold to establish even better uniform a priori bounds for all three solutions components of our approximate solutions from some large time $t_0 > 0$ onward.
The methods used for this will be a combination of testing procedures and semigroup methods used in a fairly standard bootstrap process. 
\\[0.5em]
Our first step of this section will therefore be to improve our thus far very weak bounds for the family $(n_\eps)_{\eps \in (0,1)}$ to $L^p(\Omega)$ bounds for arbitrary but finite $p$. This is done mostly by testing the first equation in (\ref{approx_system}) with $n_\eps^{p-1}$ and using the results of the previous section to argue that from some time $t_0 > 0$ onward the following is true: In any time interval of length $1$, there exists at least one time, at which $n_\eps$ is bounded in $L^p(\Omega)$, and, from that time on, the growth of $\int_\Omega n_\eps^p(\cdot, t)$ is at most exponential, whereby the bound and all the growth parameters are independent of the choice of interval. Taken together, these two facts directly imply our desired result. \\[0.5em]
This approach as well as the one used in the lemma immediately following it are again inspired by similar methods seen in \cite{WinklerSmallMass}.
\begin{lemma}
	\label{lemma:n_lp_bound}
	There exists $t_0 > 0$ such that, for each $p \in (1,\infty)$, there is $C(p) > 0$ with
	\begin{equation*}
		\|n_\eps(\cdot, t)\|_\L{p} \leq C(p)
	\end{equation*}
	for all $t > t_0$ and $\eps \in (0,1)$. 
\end{lemma} 
\begin{proof}
	Due to \Cref{lemma:things_are_small}, there exists $t_0 > 1$ such that
	\begin{equation} \label{eq:integrability_laplace_c}
		\int_{t_0 - 1}^\infty \int_\Omega\frac{|\grad n_\eps|^2}{n_\eps} \leq 1, \;\;\;\; \int_{t_0 - 1}^\infty \int_\Omega |\laplace c_\eps|^2 \leq 1
	\end{equation}
	and
	\begin{equation}\label{eq:gradient_c_smaller_1}
		\|\grad c_\eps(\cdot, t)\|_\L{2} \leq 1 \;\;\;\; \text{ for all } t > t_0 - 1
	\end{equation}
	for all $\eps \in (0,1)$.
	\\[0.5em] 
	We fix $p \in (1,\infty)$ and $t_1 > t_0$. Because 
	\[
		\int_{t_1-1}^{t_1} \int_\Omega\frac{|\grad n_\eps|^2}{n_\eps} \leq 1 \;\;\;\; \text{ for all }\eps\in(0,1),
	\]
	there must, for each $\eps\in(0,1)$, exist $t_\eps \in (t_1-1,t_1)$ such that
	\[
		\int_\Omega\frac{|\grad n_\eps(\cdot, t_\eps)|^2}{n_\eps(\cdot, t_\eps)} \leq 1.
	\] 
	Because of the Gagliardo--Nirenberg inequality, this implies 
	\begin{align*}
		\|n_\eps(\cdot, t_\eps)\|_\L{p} &= \|n^\frac{1}{2}_\eps(\cdot, t_\eps)\|^2_\L{2p} \leq K_1\left[ \|\grad n^\frac{1}{2}_\eps(\cdot, t_\eps)\|^\alpha_\L{2}\|n^\frac{1}{2}_\eps(\cdot, t_\eps)\|^{1-\alpha}_\L{2} + \|n^\frac{1}{2}_\eps(\cdot, t_\eps)\|_\L{2} \right]^2 \\
		&\leq 2K_1\left[\left( \int_\Omega \frac{|\grad n_\eps(\cdot, t_\eps)|^2}{n_\eps(\cdot, t_\eps)} \right)^{\alpha} \left( \int_\Omega n_0\right)^{1-\alpha} + \int_\Omega n_0 \right] \\
		&\leq 2K_1 \left[ \left( \int_\Omega n_0\right)^{1-\alpha} + \int_\Omega n_0 \right] =: K_2  \numberthis \label{eq:n_lp_initial_value}
	\end{align*}
	with some $K_1 > 0$ and $\alpha \defs \frac{p-1}{p} \in (0,1)$ for all $\eps \in (0,1)$.
	As our next step, a standard testing procedure yields that
	\begin{align*}
		\frac{1}{p} \frac{\d}{\d t} \int_\Omega n_\eps^p &= -(p-1)\int_\Omega |\grad n_\eps|^2 n_\eps^{p-2} + (p-1)\int_\Omega  n^{p-1}_\eps \grad  n_\eps \cdot S_\eps(\cdot, n_\eps, c_\eps) \grad c_\eps \\
		&\leq - \frac{p-1}{2}\int_\Omega |\grad n_\eps|^2 n_\eps^{p-2} + K_3\int_\Omega |\grad c_\eps|^2 n_\eps^{p} \\	&
		\leq - \frac{2(p-1)}{p^2}\int_\Omega |\grad n_\eps^\frac{p}{2}|^2 + K_3 \|\grad c_\eps\|^2_\L{4} \|n_\eps^\frac{p}{2}\|^2_\L{4} \;\;\;\; \text{ for all } t > 0 \text{ and } \eps \in (0,1)
	\end{align*}
	with $K_3 \defs \frac{p-1}{2}S^2_0(\|c_0\|_\L{\infty})$. According to the Gagliardo--Nirenberg inequality (cf.\ \cite{GNIClassic}), well-known elliptic regularity theory (cf.\ \cite[Theorem 19.1]{FriedmanPDE}), the Poincaré inequality (cf.\ \cite{BrezisFAandPDE}) and (\ref{eq:gradient_c_smaller_1}), there exists a constant $K_4 > 0$ with
	\[
		\|n_\eps^\frac{p}{2}\|^2_\L{4} \leq K_4 \|\grad n_\eps^\frac{p}{2}\|_\L{2}\|n_\eps^\frac{p}{2}\|_\L{2} + K_4\|n_\eps^\frac{p}{2}\|^2_\L{2} \;\;\;\; \text{ for all } t > 0 \text{ and } \eps \in (0,1)
	\]
	and
	\[
		\|\grad c_\eps\|^2_\L{4} \leq K_4 \| \laplace c_\eps\|_\L{2} \| \grad c_\eps \|_\L{2}  \leq K_4 \| \laplace c_\eps\|_\L{2} \;\;\;\; \text{ for all } t > t_0 - 1 \text{ and } \eps \in (0,1)
	\]
	implying that
	\begin{align*}
		\frac{1}{p} \frac{\d}{\d t} \int_\Omega n_\eps^p
		&\leq  - \frac{2(p-1)}{p^2}\int_\Omega |\grad n_\eps^\frac{p}{2}|^2 + K_3 K_4^2 \| \laplace c_\eps\|_\L{2} \left(\|\grad n_\eps^\frac{p}{2}\|_\L{2}\|n_\eps^\frac{p}{2}\|_\L{2} + \|n_\eps^\frac{p}{2}\|^2_\L{2}\right) \\
		&\leq - \frac{(p-1)}{p^2}\int_\Omega |\grad n_\eps^\frac{p}{2}|^2 + K_5 \int_{\Omega} |\laplace c_\eps|^2 \int_{\Omega} n_\eps^p + \int_{\Omega} n_\eps^p \\
		&\leq  K_5 \int_{\Omega} |\laplace c_\eps|^2 \int_{\Omega} n_\eps^p + \int_{\Omega} n_\eps^p \;\;\;\;\;\; \text{ for all } t > t_0 - 1 \text{ and } \eps \in (0,1)
	\end{align*}
	with $K_5 \defs (\frac{p^2}{4(p-1)} + \frac{1}{4}) K_3^2K_4^4$.
	This differential inequality combined with (\ref{eq:integrability_laplace_c}) and (\ref{eq:n_lp_initial_value}), then gives us that
	\[
		\int_\Omega n_\eps^p \leq K_2^p \exp\left( pK_5\int_{t_\eps}^t \int_{\Omega} |\laplace c_\eps|^2  + p(t-t_\eps)\right) \leq K_2 ^p \exp\left( pK_5  + p\right)\feds K_6 \;\;\;\; \text{ for all } t \in (t_\eps, t_\eps + 1) 
	\]
	because $t_\eps > t_1 - 1 > t_0 - 1$ by a standard comparison argument for all $\eps \in (0,1)$. As further $t_1 \in (t_\eps, t_\eps + 1)$ due to $t_\eps \in (t_1 - 1, t_1)$, this implies
	\[
		\int_\Omega n_\eps^p(\cdot, t_1) \leq K_6.
	\]
	Because $t_1 > t_0$ was arbitrary and $K_6$ is independent of $t_1$, this completes the proof.
\end{proof}
\noindent Given that we have now established quite a strong set of bounds for the family  $(n_\eps)_{\eps \in (0,1)}$, which will make the $n_\eps\grad \phi$ term in the third equation of (\ref{approx_system}) much more manageable, we will now turn our attention to said equation. 
\\[0.5em]
For this, let us briefly introduce some definitions and results used in the theory of fluid equations, which were already alluded to when talking about initial data regularity in the introduction (cf.\ (\ref{initial_data_props})) and will now become an important tool. We define $L_\sigma^p(\Omega)$ as the space of all solenoidal functions in $(L^p(\Omega))^2$, or more precisely
\[
	L^p_\sigma(\Omega) \defs \left\{\, f\in(\L{p})^2 \,\mid \, \div f = 0 \, \right\}
\]
for all $p \in (1,\infty)$ with $\div$ interpreted as a distributional derivative. As proven in e.g.\ \cite{GaldiNavierStokes1}, there then exists a unique, continuous projection 
\[ 
	\mathcal{P}_p: (L^p(\Omega))^2 \rightarrow L_\sigma^p(\Omega)
\]
called the Helmholtz projection for all $p \in (1,\infty)$. In fact, $\mathcal{P}_2$ is an orthogonal projection (cf.\ \cite[II.2.5]{MR1928881}).
\\[0.5em]
Using this, we then define the Stokes operator on $L^p_\sigma(\Omega)$ as 
\[
	A_p \defs -\mathcal{P}_p \laplace
\] with $D(A_p) \defs W^{2,p}_{0,\sigma}(\Omega) \defs (W^{2,p}_{0}(\Omega))^2 \cap L_\sigma^p(\Omega)$  (cf.\ \cite{MR605289}, \cite{MR1928881}) for all $p \in (1,\infty)$. In \cite{MR605289} and \cite{GigaDomainsFractionalPowers1985}, it is then shown that, for all $p\in (1,\infty)$, $A_p$ is sectorial (in fact its spectrum is contained in $(0,\infty)$), that $-A_p$ generates a bounded analytic semigroup $(e^{-t A_p})_{t \geq 0}$ of class $C_0$ on $D(A_p)$ and that the fractional powers $A_p^\alpha$ of $A_p$ exist for all $\alpha \in (0,1)$. Due to the regularity theory for the stationary Stokes equation (cf.\ e.g.\ \cite[Lemma IV.6.1]{GaldiNavierStokes1}), the Stokes operator further has the following property: For each $p\in(1,\infty)$, there exists $C(p) > 0$ such that 
\begin{equation}\label{eq:stokes_op_regularity}
	\|\varphi\|_{W^{2,p}(\Omega)} \leq C(p)\|A_p \varphi\|_\L{p} \;\;\;\;\text{ for all } \varphi \in D(A_p).
\end{equation}
Notably, this means that the norm $\|A_p \cdot \|_\L{p}$ is equivalent to the standard Sobolev norm $\|\cdot \|_{W^{2,p}(\Omega)}$ on $D(A_p)$ for all $p\in(1,\infty)$. As such, we will from hereon out consider $\|A_p\cdot\|_\L{p}$ to be the default norm of the space $D(A_p)$. In a similar vein when talking about the domains of the fractional powers $D(A_p^\alpha)$, $p \in (1,\infty)$, $\alpha \in (0,1)$, we will from now on always assume these spaces to be equipped with the corresponding norm $\|A_p^\alpha \cdot \|_\L{p}$. Framed in this way, the spaces $D(A_p^\alpha)$ then have rather favorable continuous embedding properties into certain Sobolev and Hölder spaces due to standard semigroup theory (cf.\ \cite{GigaDomainsFractionalPowers1985} or \cite[p.~39]{DanHenryGeom}) and the regularity property (\ref{eq:stokes_op_regularity}), which will be useful on multiple occasions.
\\[0.5em]
By revisiting the construction of the Helmholtz projection in \cite{GaldiNavierStokes1}, which rests on essentially solving a certain elliptic Neumann problem, we see that for sufficiently regular functions (e.g.\ $C^2$) all Helmholtz projections and Stokes operators introduced above are in fact identical and as such we will often just write $\mathcal{P}$ and $A$ for the projection and operator, respectively, where appropriate.
\\[0.5em]
Let us now return to our actual objective, namely the derivation of an $L^2(\Omega)$ bound for the gradients of the family $(u_\eps)_{\eps \in (0,1)}$. Structurally this proof is very similar to the one above in that we again establish boundedness for the gradients at one time in every time interval of length $1$ (from some time $t_0 > 0$ onward) and then derive an additional growth restriction in said interval, whereby both times all parameters are again independent of the choice of interval.
\begin{lemma}
	\label{lemma:grad_u_bound}
	There exist $t_0 > 0$ and $C > 0$ such that 
	\[
		\int_\Omega |\grad u_\eps(\cdot, t)|^2 \leq C 	
	\]
	for all $t > t_0$ and $\eps \in (0,1)$.
\end{lemma}
\begin{proof}
	We first fix $t_0 > 1$ and $K_1 > 0$ such that
	\begin{equation} \label{eq:n_u_bounds}
		\|n_\eps(\cdot, t)\|_\L{2} \leq K_1, \;\;\;\; \|u_\eps(\cdot, t)\|_\L{2} \leq 1 \;\;\;\;\;\;\;\; \text{ for all } t > t_0 - 1 
	\end{equation}
	and
	\begin{equation} \label{eq:grad_integrability}
		\int_{t_0 - 1}^\infty \int_\Omega |\grad u_\eps|^2 \leq 1
	\end{equation}
	for all $\eps \in (0,1)$ according to \Cref{lemma:n_lp_bound} and \Cref{lemma:things_are_small}. 
	\\[0.5em]
	Let $t_1 > t_0$ be fixed, but arbitrary. Then for each $\eps \in (0,1)$, there exists $t_\eps \in (t_1-1, t_1)$ such that
	\begin{equation} \label{eq:grad_small_at_one_point}
		\int_\Omega |\grad u_\eps(\cdot, t_\eps)|^2 \leq 1  
	\end{equation}
	due to (\ref{eq:grad_integrability}).
	We may now further fix $K_2 > 0$ such that
	\begin{equation}
		\|u_\eps\|^2_\L{\infty} \leq K_2 \|A u_\eps\|_\L{2} \| u_\eps\|_\L{2} \leq K_2 \|A u_\eps\|_\L{2}  \;\;\;\; \text{ for all } t > t_0 - 1 \text{ and } \eps \in (0,1) \label{eq:grad_u_initial_value}
	\end{equation}
	due to (\ref{eq:n_u_bounds}), the Gagliardo--Nirenberg inequality and the regularity property in (\ref{eq:stokes_op_regularity}).
	We then apply the Helmholtz projection to the third equation in (\ref{approx_system}) and test with $A u_\eps$ to see that
	\begin{align*}
		\frac{1}{2}\frac{\d}{\d t}\int_\Omega |\grad u_\eps|^2 = \frac{1}{2}\frac{\d}{\d t}\int_\Omega |A^\frac{1}{2} u_\eps|^2 &= -\int_\Omega |Au_\eps|^2 -\int_\Omega Au_\eps \cdot \mathcal{P}\left[ (u_\eps \cdot \grad) u_\eps \right] + \int_\Omega Au_\eps \cdot \mathcal{P}\left[ n_\eps  \grad \phi \right] \\
		&\leq -\frac{1}{2}\int_\Omega |Au_\eps|^2 + \int_\Omega \left|\mathcal{P}\left[ (u_\eps \cdot \grad) u_\eps \right] \right|^2+ \int_\Omega \left|\mathcal{P} \left[ n_\eps  \grad \phi \right] \right|^2 \\
		&\leq -\frac{1}{2}\int_\Omega |Au_\eps|^2 + \int_\Omega \left| (u_\eps \cdot \grad) u_\eps \right|^2 + \int_\Omega \left| n_\eps  \grad \phi \right|^2 \\
		&\leq -\frac{1}{2}\int_\Omega |Au_\eps|^2 + \|u_\eps\|^2_\L{\infty} \|\grad u_\eps\|^2_\L{2} + K_1^2 \|\grad \phi\|^2_\L{\infty} \\
		&\leq -\frac{1}{2}\int_\Omega |Au_\eps|^2 + K_2 \|A u_\eps\|_\L{2} \|\grad u_\eps\|^2_\L{2} + K_1^2 \|\grad \phi\|^2_\L{\infty} \\
		&\leq \frac{K_2^2}{2}\left(\int_\Omega |\grad u_\eps|^2 \right)^2 + K_1^2 \|\grad \phi\|^2_\L{\infty} 
	\end{align*}
	for all $t > t_0 - 1$ and $\eps \in (0,1)$ by using (\ref{eq:grad_u_initial_value}), Young's inequality and some fundamental properties of the fractional powers of the Stokes operator (cf.\ \cite{DanHenryGeom}, \cite[Lemma III.2.2.1]{MR1928881}). This further  implies that
	\[
		\frac{\d}{\d t}\int_\Omega |\grad u_\eps|^2 \leq K_3\left(\int_\Omega |\grad u_\eps|^2 \right)^2 + K_4 \;\;\;\; \text{ for all } t > t_0 - 1 \text{ and } \eps \in (0,1)
	\]
	with $K_3 \defs K_2^2$, $K_4 \defs 2K_1^2\|\grad \phi\|^2_\L{\infty}$. This differential inequality combined with (\ref{eq:grad_small_at_one_point}) and  (\ref{eq:grad_integrability}) then gives us that
	\begin{align*}
		\int_\Omega |\grad u_\eps(\cdot, t)|^2 &\leq \left( \int_\Omega |\grad u_\eps(\cdot, t_\eps)|^2 \right)\exp\left( K_3\int_{t_\eps}^{t} \int_\Omega |\grad u_\eps|^2 \right)  + K_4 \int_{t_\eps}^{t} \exp\left( K_3\int_{s}^{t} \int_\Omega |\grad u_\eps|^2 \right) \d s \\
		&\leq (1+K_4) e^{K_3} \feds K_5 \;\;\;\; \text{ for all } t \in (t_\eps, t_\eps + 1) \text{ and } \eps \in (0,1) 
	\end{align*}
	because $t_\eps > t_1 - 1 > t_0 - 1$ by standard comparison argument. As further $t_1 \in (t_\eps, t_\eps + 1)$ due to the fact that $t_\eps \in (t_1 - 1, t_1)$, this implies
	\[
		 \int_\Omega |\grad u_\eps(\cdot, t_1)|^2 \leq K_5.
	\]
	Given that $t_1 > t_0$ was arbitrary and $K_5$ is independent of $t_1$, this completes the proof.
\end{proof}
\noindent As already seen in the proof above, deriving an $L^2(\Omega)$ bound for the gradients of the family $(u_\eps)_{\eps\in(0,1)}$ is equivalent to deriving a bound for $\|A^\frac{1}{2} u_\eps(\cdot, t)\|_\L{2}$, which is fairly easy to check by using fundamental properties of the fractional powers of the Stokes operator (cf.\ \cite{DanHenryGeom}, \cite[Lemma III.2.2.1]{MR1928881}). As our next step then, we now want to expand on this by proving stronger bounds of the form $\|A^\beta u_\eps(\cdot, t)\|_\L{p} \leq C$ because the embedding properties of $D(A_2^\frac{1}{2})$ are not quite sufficient for our later arguments, namely the derivation of certain Hölder-type bounds. This is done mostly by using the above results in combination with semigroup methods and well-known smoothing properties of the Stokes semigroup (cf.\ \cite{GigaSolutionsSemilinearParabolic1986}, \cite[p.~26]{DanHenryGeom}).
\begin{lemma}
	\label{lemma:Au_beta_bound}
	There exists $t_0 > 0$ such that, for each $\beta \in (\frac{1}{2},1)$ and $p \in (2,\infty)$, there is $C(\beta,p) > 0$ with
	\[
		\| A^\beta u_\eps(\cdot, t) \|_\L{p} \leq C(\beta, p)
	\]
	for all $t > t_0$ and $\eps \in (0,1)$.
\end{lemma}
\begin{proof}
	We fix $\beta \in (\frac{1}{2},1)$ and $p \in (2,\infty)$. We then fix a time $t_0 > 1$ independently of $p$ and a constant $K_1 > 0$ such that
	\[
		\|u_\eps(\cdot, t)\|_\L{2} \leq 1, \;\;\;\; \|n_\eps(\cdot, t)\|_\L{p}\leq K_1, \;\;\;\;	\|\grad u_\eps(\cdot, t)\|_\L{2}\leq K_1 \;\;\;\;  \text{ for all } t > t_0 - 1 \text{ and } \eps \in (0,1)
	\]
	according to \Cref{lemma:things_are_small}, \Cref{lemma:n_lp_bound} and \Cref{lemma:grad_u_bound}. Lastly, we fix $q \in (2,p)$ such that
	\begin{equation} \label{eq:p_q_parameter_ineq}
		\beta + \frac{1}{q} - \frac{1}{p} < 1.
	\end{equation}
	The Sobolev embedding theorem (cf.\ \cite[Theorem 2.72]{EllipticFunctionSpaces}) implies that there further exists $K_2 > 0$ such that
	\[
		\|u_\eps(\cdot, t)\|_\L{\frac{q(p+q)}{p-q}} \leq K_2, \;\;\;\; \|u_\eps(\cdot, t)\|_\L{p}\leq K_2\;\;\;\;  \text{ for all } t > t_0 - 1 \text{ and } \eps \in (0,1)
	\]
	due to us working in two dimensions and due to the previously established bounds. 
	\\[0.5em]
	Let now $t_1 > t_0$ be fixed, but arbitrary.
	Then relying on the smoothing and continuity properties of the Stokes semigroup $(e^{-tA})_{t\geq 0}$ and the Helmholtz projection $\mathcal{P}_p$ (cf.\ \cite{GaldiNavierStokes1}, \cite[p.~201]{GigaSolutionsSemilinearParabolic1986}, \cite[p.~26]{DanHenryGeom}), we estimate each $u_\eps$ using the variation-of-constant representation of the third equation in (\ref{approx_system}) on $(t_1 - 1, t_1)$ after projecting with $\mathcal{P}$ as follows:
 	\begin{align*}
 	&\| \, A^{\beta}u_\eps(\cdot, t) \|_\L{p} \\
	&= \left\| A^{\beta}e^{-(t-(t_1 - 1))A}u_\eps(\cdot, t_1-1) - \int_{t_1-1}^t A^\beta e^{-(t-s)A} \mathcal{P}\left[ (u_\eps(\cdot, s) \cdot \grad) u_\eps(\cdot, s)  \right] \d s \right.  \\
 	&\;\;\;\;\;\;\;\;\left.+ \int_{t_1-1}^t A^\beta e^{-(t-s)A} \mathcal{P}\left[ n_\eps(\cdot, s) \grad \phi  \right] \d s \, \right\|_\L{p} \\
 	&\leq K_3(t-(t_1-1))^{-\beta}\|u_\eps(\cdot, t_1 - 1)\|_\L{p} + K_3\int_{t_1-1}^t (t-s)^{-\beta - \frac{1}{q} + \frac{1}{p}} \left\|\mathcal{P} \left[(u_\eps(\cdot, s) \cdot \grad) u_\eps(\cdot, s)\right]  \right\|_\L{q} \d s \\
 	&\;\;\;\;\;\;\;\;+ K_3\int_{t_1-1}^t (t-s)^{-\beta}\left\| \mathcal{P}\left[ n_\eps(\cdot, s) \grad \phi  \right] \right\|_\L{p} \d s \\
 	&\leq K_5(t-(t_1-1))^{-\beta} + K_5 + K_3 K_4\int_{t_1-1}^t (t-s)^{-\beta - \frac{1}{q} + \frac{1}{p}} \|u_\eps(\cdot, s)\|_\L{\frac{q(p+q)}{p-q}} \| \grad u_\eps(\cdot, s) \|_\L{\frac{p+q}{2}} \d s \\
 	&\leq K_5 (t-(t_1-1))^{-\beta} + K_5 + K_2 K_3 K_4\int_{t_1-1}^t (t-s)^{-\beta - \frac{1}{q} + \frac{1}{p}} \| \grad u_\eps(\cdot, s) \|_\L{\frac{p+q}{2}} \d s \numberthis \label{eq:Au_1}
 	\end{align*}
 	with some $K_3, K_4 > 0$ and $K_5 \defs \max( K_2K_3, \frac{K_1 K_3 K_4 \|\grad \phi\|_\L{\infty}}{1 - \beta})$ for all $t\in(t_1 - 1, t_1]$ and $\eps \in (0,1)$.
 	\\[0.5em]
 	Interpolation using the Hölder inequality combined with the fact that $D(A_p^\beta)$ embeds continuously into $W^{1,p}(\Omega)$ (cf.\ \cite[p.~39]{DanHenryGeom} or \cite{GigaDomainsFractionalPowers1985}) because $\beta > \frac{1}{2}$ then gives us $K_6 > 0$ such that
 	\begin{align*}
	 	\|\grad u_\eps(\cdot, t)\|_\L{\frac{p+q}{2}} &\leq \|\grad u_\eps(\cdot, t)\|^{1-\alpha}_\L{2}\|\grad u_\eps(\cdot, t)\|^\alpha_\L{p} \leq K_1^{1-\alpha} K_6 \|A^\beta u_\eps(\cdot, t)\|^\alpha_\L{p}
 	\end{align*}
 	for all $t > t_0 - 1$ and $\eps \in (0,1)$ with $\alpha \defs ( \frac{1}{2} - \frac{2}{p+q} ) ( \frac{1}{2} - \frac{1}{p} )^{-1} \in (0,1)$. In (\ref{eq:Au_1}), this then results in
 	\[
	 	\| A^{\beta}u_\eps(\cdot, t) \|_\L{p} \leq  K_5(t-(t_1-1))^{-\beta} + K_5 + K_7 \int_{t_1-1}^t (t-s)^{-\beta - \frac{1}{q} + \frac{1}{p}} \|A^\beta u_\eps(\cdot, s)\|^\alpha_\L{p} \d s\label{eq:Au_2} \numberthis
 	\]
 	for all $t\in(t_1 - 1, t_1]$ and $\eps \in (0,1)$ with $K_7 \defs K_1^{1-\alpha} K_2 K_3 K_4 K_6$.
	\\[0.5em]
	Let now \[
		M_\eps \defs  \sup_{s\in(t_1 - 1, t_1]} (s - (t_1-1))^\beta \| A^{\beta}u_\eps(\cdot, s) \|_\L{p} < \infty.
	\]
	for all $t\in(t_1 - 1, t_1]$ and $\eps \in (0,1)$. With this definition, we can conclude from (\ref{eq:Au_2}) that
	\[
		(t - (t_1-1))^\beta\| A^{\beta}u_\eps(\cdot, t) \|_\L{p} \leq  2K_5 + K_7M_\eps^\alpha(t - (t_1-1))^\beta\int_{t_1-1}^t (t-s)^{-\beta - \frac{1}{q} + \frac{1}{p}} (s - (t_1-1))^{-\alpha \beta} \d s\label{eq:Au_3} \numberthis
 	\]
	for all $t\in(t_1 - 1, t_1]$ and $\eps \in (0,1)$. As both of the exponents in the remaining integral are greater than $-1$ due to our choice of $q$, a straightforward estimation yields $K_8 > 0$ such that 
	\begin{align*}
		\int_{t_1-1}^t (t-s)^{-\beta - \frac{1}{q} + \frac{1}{p}} (s - (t_1-1))^{-\alpha \beta} \d s \leq K_8 (t - (t_1 - 1))^{1 - \beta  - \frac{1}{q} + \frac{1}{p} - \alpha \beta}
	\end{align*}
	and thus 
	\[
		(t - (t_1-1))^\beta\int_{t_1-1}^t (t-s)^{-\beta - \frac{1}{q} + \frac{1}{p}}(s - (t_1-1))^{-\alpha \beta} \d s \leq K_8
	\]
	for all $t\in(t_1 - 1, t_1]$ and $\eps \in (0,1)$ again due to our choice of $q$. If we apply this to (\ref{eq:Au_3}), we find that 
	\[
		M_\eps \leq 2K_5 + K_7K_8 M_\eps^\alpha \leq 2K_5 + (1-\alpha)(K_7K_8)^\frac{1}{1-\alpha} + \alpha M_\eps
	\]
	due to Young's inequality and therefore that
	\[
		M_\eps \leq \frac{2K_5}{1-\alpha} + (K_7K_8)^\frac{1}{1-\alpha} \sfed K_9
	\]
	for all $\eps \in (0,1)$. This further gives us that
	\[
		\|A^\beta u_\eps(\cdot, t_1)\|_\L{p} \leq M_\eps \leq K_9 \;\;\;\; \text{ for all } \eps \in (0,1)
	\]
	and thus completes the proof as $t_1 > t_0$ was arbitrary.
\end{proof}
\noindent Given the regularity properties of the fractional powers of the Stokes operator, we gain the following corollary, which translates the above abstract bounds into more familiar settings.
\begin{corollary} \label{corollary:u_linfty_bound}
	There exists $t_0 > 0$ such that, for each $\alpha \in (0,1)$, there is $C(\alpha) > 0$ such that
	\[
		\|u_\eps(\cdot, t)\|_{L^{\infty}(\Omega)} \leq \|u_\eps(\cdot, t)\|_{W^{1,\infty}(\Omega)}  \leq \|u_\eps(\cdot, t)\|_{C^{1+\alpha}(\overline{\Omega})} \leq C(\alpha)
	\]
	for all $t > t_0$ and $\eps \in (0,1)$.
\end{corollary}
\begin{proof}
	We first choose $p \in (2,\infty)$ and $\beta \in(\frac{1}{2}, 1)$ such that 
	\[
		1 + \alpha < 2\beta - \frac{2}{p}, 
	\]
	which is always possible.
	\\[0.5em]
	\Cref{lemma:Au_beta_bound} then gives us $t_0 > 0$ and $K_1 > 0$ such that
	\[
		\|A^\beta u_\eps(\cdot, t)\|_\L{p} \leq K_1
	\]
	for all $t > t_0$. Note that the $t_0$ given to us by the lemma is in fact independent of $\beta$ and $p$.
	By well-known continuous embedding property of $D(A_p^\beta)$ into $C^{1+\alpha}(\overline{\Omega})$ seen for example in \cite{GigaDomainsFractionalPowers1985} or \cite[p.~39]{DanHenryGeom}, this already implies our desired result.
\end{proof}
\noindent
As our next step, we will now use semigroup methods to prove some additional  bounds for the families $(n_\eps)_{\eps \in (0,1)}$ and $(c_\eps)_{\eps \in (0,1)}$ in a fairly standard and therefore brief fashion:
\begin{lemma}
There exists $t_0 > 0$ such that, for each $p \in (1,\infty]$, there is $C(p) > 0$ with
\[
	\|\grad c_\eps(\cdot, t) \|_\L{p} \leq C(p)
\]
for all $t > t_0$ and $\eps \in (0,1)$.
\label{lemma:grad_c_bound}
\end{lemma}
\begin{proof}
	We use the variation-of-constant representation combined with semigroup smoothness estimates from \cite[Lemma 1.3]{WinklerSemigroupRegularity} to estimate the family $(c_\eps)_{\eps \in (0,1)}$ as follows:
	\begin{align*}
		&\left\|\grad c_\eps(\cdot, t)\right\|_\L{p}\\ \leq& \left\| \; \grad e^{\laplace}c_\eps(\cdot, t-1) - \int_{t-1}^t \grad e^{(t-s)\laplace}(u_\eps\cdot\grad c_\eps) \d s - \int_{t-1}^t \grad e^{(t-s)\laplace}f(c_\eps)n_\eps \d s \; \right\|_\L{p} \\
		\leq&\, K_1\|c_0\|_\L{p} + K_1\int_{t-1}^t (1+(t-s)^{-1 + \frac{1}{p}}) \left\{ \|u_\eps\|_\L{\infty} \|\grad c_\eps\|_\L{2} + \|f\|_{L^\infty([0, \|c_0\|_\L{\infty}])} \|n_\eps\|_\L{2} \right\} \d s
	\end{align*}
	with some constant $K_1 > 0$ for all $t > 1$.
	This already implies our desired result for all finite $p$ and sufficiently large $t_0 > 0$ because of \Cref{corollary:u_linfty_bound}, \Cref{lemma:things_are_small} and \Cref{lemma:n_lp_bound}. For the case $p = \infty$, we use a similar approach as before to derive
	\begin{align*}
		&\|\grad c_\eps(\cdot, t)\|_\L{\infty} \\
		\leq& \, K_2\|c_0\|_\L{\infty} + K_2\int_{t-1}^t (1+(t-s)^{-\frac{3}{4}}) \left\{ \|u_\eps\|_\L{\infty} \|\grad c_\eps\|_\L{4} + \|f\|_{L^\infty([0, \|c_0\|_\L{\infty}])} \|n_\eps\|_\L{4} \right\} \d s
	\end{align*}
	with some constant $K_2 > 0$ for all $t > 1$. Using this very result for $p = 4$ as proven above and the same lemmas as before then completes the proof.
\end{proof}

\begin{lemma} \label{lemma:n_linfty_bound}
	There exists $t_0 > 0$ and $C > 0$ such that
	\[
		\| n_\eps(\cdot, t) \|_\L{\infty} \leq C
	\]
	for all $t > t_0$ and $\eps \in (0,1)$.
\end{lemma}
\begin{proof}
	Using the variation-of-constants representation applied to the family $(n_\eps)_{\eps \in (0,1)}$ and combining it with semigroup smoothness estimates from \cite[Lemma 1.3]{WinklerSemigroupRegularity} yields that
	\begin{align*}
		&\left\|n_\eps(\cdot, t)\right\|_\L{\infty} \\
		=&\left \| \; e^{\laplace}n_\eps(\cdot, t - 1) - \int_{t-1}^t  e^{(t-s)\laplace} (u_\eps \cdot \grad n_\eps)  \d s - \int_{t-1}^t e^{(t-s)\laplace} \div (n_\eps S_\eps(\cdot, n_\eps, c_\eps) \grad c_\eps) \d s \; \right\|_\L{\infty}
		 \\
		=&\left \| \; e^{\laplace}n_\eps(\cdot, t - 1) - \int_{t-1}^t  e^{(t-s)\laplace} \div (u_\eps n_\eps)  \d s - \int_{t-1}^t e^{(t-s)\laplace} \div (n_\eps S_\eps(\cdot, n_\eps, c_\eps) \grad c_\eps) \d s \; \right\|_\L{\infty} \\
		\leq&\,  K_1\|n_\eps(\cdot, t - 1)\|_\L{p} +
		\\ & \;\;\;\;\; K_1 \int_{t-1}^t (1 + (t-s)^{-\frac{3}{4}}) \left\{ \; \| u_\eps\|_\L{\infty}\|n_\eps\|_\L{4} + S_0(\|c_0\|_\L{\infty})\|n_\eps\|_\L{8}   \|\grad c_\eps\|_\L{8}  \; \right\}  \d s
	\end{align*}
	with some constant $K_1 > 0$ for all $t > 1$ because $\div u_\eps = 0$. Combining this with \Cref{corollary:u_linfty_bound}, \Cref{lemma:n_lp_bound} and \Cref{lemma:grad_c_bound} then gives us the desired bound by similar arguments as in the previous lemma.
\end{proof}
\noindent One immediate consequence of this lemma is an additional global space time integrability property for the gradients of $n_\eps$. This property will prove useful when later arguing that $n$ is in fact a weak solution of its associated differential equation as a step in the process of proving its more classical solution properties.
\begin{lemma}\label{lemma:grad_n_integrability}
	There exists $t_0 > 0$ and $C > 0$ such that
	\[
		\int_{t_0}^\infty \int_\Omega |\grad n_\eps|^2 \leq C
	\]
	for all $\eps \in (0,1)$.
\end{lemma}
\begin{proof}
	Due to \Cref{lemma:n_linfty_bound}, there exist $t_0 > 0$ and $K_1 > 0$ such that
	\begin{equation}\label{eq:grad_n_linfty_bound}
		\|n_\eps(\cdot, t)\|_\L{\infty} \leq K_1 \;\;\;\; \text{ for all } t \geq t_0 \text{ and } \eps \in (0,1). 
	\end{equation}
	Testing the first equation in (\ref{approx_system}) with $n_\eps$ in a similar fashion as in the proof of \Cref{lemma:n_lp_bound} and applying Young's inequality immediately yields
	\[
		\frac{1}{2} \frac{\d}{\d t} \int_\Omega n^2_\eps \leq - \frac{1}{2}\int_\Omega |\grad n_\eps|^2 + \frac{S_0^2(\|c_0\|_\L{\infty})}{2} \int_\Omega |\grad c_\eps|^2 n_\eps^2 \leq - \frac{1}{2}\int_\Omega |\grad n_\eps|^2 + K_2 \int_\Omega |\grad c_\eps|^2
	\]
	with $K_2 \defs \frac{1}{2}K_1^2 S_0^2(\|c_0\|_\L{\infty})$ for all $t > t_0$ and $\eps \in (0,1)$. Time integration combined with \Cref{lemma:basic_props} and (\ref{eq:grad_n_linfty_bound}) then directly gives us our desired result.
\end{proof}

\subsection{Establishing baseline parabolic Hölder bounds for $n_\eps$, $c_\eps$ and $u_\eps$}

As our next step in the journey towards a proof of \Cref{theorem:eventual_smoothness}, we will now transition from only establishing uniform space bounds for the solution components of our approximate solutions to full parabolic Hölder-type bounds. Establishing such bounds  will then allow us to use the well-known Arzelà--Ascoli compact embedding theorem to argue that, at least from some large time onward, the generalized solutions constructed in \Cref{theorem:ext_existence} were in fact of a similarly high level of regularity.
\\[0.5em]
We will start by establishing a $C^{\alpha}([t,t+1];C^{1+\alpha}(\overline{\Omega}))$-type bound for the family $(u_\eps)_{\eps \in (0,1)}$ as $u_\eps$ plays a role in all three equations of (\ref{approx_system}) due to the convection terms. While the step from the bounds already established in the previous sections to uniform Hölder bounds is often taken care of by employing well-known and ready-made parabolic regularity theory, we are not aware of such a result that fits the third equation in (\ref{approx_system}) and gives us the type of bound desired here. 
\\[0.5em]
Similar to the methods seen in e.g.\ \cite{FriedmanPDE} and \cite[Lemma 3.4]{MR3801284}, we will therefore use a different approach that is based on the regularity properties of the fractional powers of the Stokes operator and the variation-of-constants representation of the family $(u_\eps)_{\eps \in (0,1)}$, not unlike what we did in the proof of \Cref{lemma:Au_beta_bound}. The key difference to similar previous efforts in this paper is that we apply these methods to difference terms of the form $u_\eps(\cdot, t_2) - u_\eps(\cdot, t_1)$ instead of $u_\eps$ itself.
\\[0.5em]
We begin by proving an analogue to \Cref{lemma:Au_beta_bound} for such difference terms.

\begin{lemma}\label{lemma:time_hoelder_proto}
There exists $t_0 > 0$ such that, for each $\beta \in (0,1)$, $p\in (2,\infty)$, there is a constant $C(\beta, p) > 0$ such that
\[
	\|A^\beta \left[\, u_\eps(\cdot, t_2) - u_\eps(\cdot, t_1) \,\right] \|_\L{p} \leq C(\beta, p) (t_2 - t_1)^\frac{1-\beta}{2}
\]
for all $t_2 > t_1 > t_0$ with $t_2 - t_1 < 1$ and $\eps \in (0,1)$.
\end{lemma}
\begin{proof}
We fix $\beta \in (0,1)$ and $p \in (2,\infty)$. We then further fix $t_0 > 0$ independently of $\beta$ and $p$ and $K_1 > 0$ such that
\begin{equation*}
	\|n_\eps(\cdot, t)\|_\L{p} \leq K_1, \;\; \|u_\eps(\cdot, t)\|_{L^\infty(\Omega)} \leq K_1, \;\; \|\grad u_\eps(\cdot, t)\|_{L^\infty(\Omega)} \leq K_1, \;\; \|A^\frac{\beta + 1}{2} u_\eps(\cdot, t)\|_{L^p(\Omega)} \leq K_1
\end{equation*}
for all $t \geq t_0$ and $\eps \in (0,1)$ according to \Cref{lemma:n_lp_bound}, \Cref{corollary:u_linfty_bound} and \Cref{lemma:Au_beta_bound}. This then implies for 
\[
	F_\eps(x,t) \defs \mathcal{P} \left[\, - (u_\eps(x,t)\cdot \grad) u_\eps(x,t) + \grad \phi(x) \cdot n_\eps(x,t) \, \right] \;\;\;\; \text{ for all } (x,t) \in \Omega\times[0,\infty) 
\]
that
\begin{align*}
	\|F_\eps(\cdot, t)\|_\L{p} 
	&\leq K_2 \left\| -(u_\eps(\cdot,t)\cdot \grad) u_\eps(\cdot,t) + \grad \phi \cdot n_\eps(\cdot,t) \, \right\|_\L{p} \\
	&\leq K_2|\Omega|^{\frac{1}{p}}\|u_\eps(\cdot, t)\|_\L{\infty}\|\grad u_\eps(\cdot, t)\|_\L{\infty} + K_2 \|\grad \phi\|_\L{\infty}\|n_\eps(\cdot, t)\|_\L{p} \\
	&\leq |\Omega|^\frac{1}{p}K_1^2 K_2 + K_1 K_2\|\grad \phi\|_\L{\infty} \feds K_3
\end{align*}
for all $t > t_0$ and $\eps \in (0,1)$ with some constant $K_2 > 0$ due to the continuity of the Helmholtz projection $\mathcal{P}_p$.
\\[0.5em]
Let now $t_2 > t_1 > t_0$ be such that $t_2 - t_1 < 1$. Using the variation-of-constants representation of $u_\eps$ with respect to the semigroup $(e^{-tA})_{t\geq 0}$, we then observe that
\begin{align*}
	&\left\|A^\beta \left[\, u_\eps(\cdot, t_2) - u_\eps(\cdot, t_1)\,\right]\right\|_\L{p} \\
	\leq& \left\| A^\beta e^{-(t_2-t_0) A}u(\cdot, t_0) - A^\beta e^{-(t_1-t_0)A}u(\cdot, t_0)  \right\|_\L{p} \\ +& \left\|\int_{t_0}^{t_2}A^{\beta} e^{-(t_2 - s)A}F_\eps(\cdot, s) \d s - \int_{t_0}^{t_1}A^{\beta} e^{-(t_1 - s)A}F_\eps(\cdot, s) \d s \right\|_\L{p} \\
	\feds& \;D_1 + D_2 \;\;\;\; \text{ for all } \eps \in (0,1).
\end{align*}
Using well-known smoothing properties of the Stokes semigroup (cf.\ \cite[Theorem 1.4.3]{DanHenryGeom}) combined with the defining fact that 
\[	
\tfrac{\d}{\d t} e^{-tA}\varphi = -A e^{-tA}\varphi \;\;\;\; \text{ for all } t > 0 \text{ and } \varphi \in C^2(\overline{\Omega}) \text{ with } u = 0 \text{ on }\partial\Omega \text{ and }\div u = 0 \text{ on } \Omega
\]
and the fundamental theorem of calculus, we now estimate $D_1$ as follows:
\begin{align*}
	D_1 &= \left\| A^\beta \int_{t_1}^{t_2} Ae^{-(s-t_0)A}u(\cdot, t_0) \d s \, \right\|_\L{p}
	= \left\| \int_{t_1}^{t_2} A^{\frac{\beta + 1}{2}}e^{-(s-t_0)A}A^\frac{\beta + 1}{2}u(\cdot, t_0) \d s \,  \right\|_\L{p} \\
	&\leq K_4 \int_{t_1}^{t_2} (s-t_0)^{-\frac{\beta + 1}{2}} \|A^{\frac{\beta + 1}{2}}u(\cdot, t_0)\|_\L{p} \d s \leq K_1 K_4 \int_{t_1}^{t_2} (s-t_0)^{-\frac{\beta + 1}{2}} \d s  \\
	&= \frac{2 K_1 K_4}{1-\beta} ((t_2-t_0)^\frac{1 - \beta}{2} - (t_1-t_0)^\frac{1-\beta}{2} ) \leq K_5 (t_2 - t_1 )^\frac{1-\beta}{2} \label{eq:holder_d1_estimate} \;\;\;\; \text{ for all } \eps \in (0,1)\numberthis
\end{align*}
with $K_4 > 0$ being the smoothing constant from \cite[Theorem 1.4.3]{DanHenryGeom} and $K_5 \defs \frac{2 K_1 K_4}{1-\beta}$.
\\[0.5em]
By a similar argument, we gain for $D_2$ that
\begin{align*}
	D_2 &\leq K_6\int_{t_1}^{t_2} (t_2-s)^{-\beta} \|F_\eps(\cdot, s)\|_\L{p} \d s + \left\|\int_{t_0}^{t_1} A^{\beta}\left[e^{-(t_2 - s)A} - e^{-(t_1 - s)A} \right]F_\eps(\cdot, s) \d s \right\|_\L{p} \\
	&\leq \frac{K_3 K_6}{1-\beta}  (t_2 - t_1)^{1-\beta} +  \left\|\int_{t_0}^{t_1} A^\beta \int_{t_1}^{t_2} A e^{-(\sigma - s)A} F_\eps(\cdot, s) \d \sigma \d s \right\|_\L{p} \\
	&\leq \frac{K_3 K_6}{1-\beta}  (t_2 - t_1)^{1-\beta} +  K_3 K_6 \int_{t_0}^{t_1}\int_{t_1}^{t_2} (\sigma - s)^{-1 - \beta}  \d \sigma \d s \\
	&= \frac{K_3 K_6}{1-\beta}  (t_2 - t_1)^{1-\beta} - \frac{K_3 K_6}{\beta} \int_{t_0}^{t_1} (t_2 - s)^{-\beta} - (t_1 - s)^{-\beta}   \d s \\
	&=  \frac{K_3 K_6}{1-\beta}  (t_2 - t_1)^{1-\beta}  + \frac{K_3 K_6}{\beta(1 - \beta)} \left[ (t_2 - t_1)^{1-\beta} - (t_2 - t_0)^{1-\beta} - (t_1 - t_1)^{1-\beta} + (t_1 - t_0)^{1-\beta} \right] \\
	&\leq K_7 (t_2 - t_1)^{1-\beta}  \label{eq:holder_d2_estimate} \numberthis
\end{align*}
for all $\eps \in (0,1)$ with $K_6 > 0$ being the smoothing constant from \cite[Theorem 1.4.3]{DanHenryGeom} and $K_7 \defs K_3 K_6\frac{1+\beta}{\beta(1-\beta)}$. Note here that the last step was made possible by the fact that $(t_2-t_0)^{1-\beta} \geq (t_1-t_0)^{1-\beta}$. 
\\[0.5em]
Because $t_2 - t_1 < 1$, we further know that
\[
	(t_2 - t_1)^{1-\beta} = (t_2 - t_1)^{\frac{1-\beta}{2}} (t_2 - t_1)^{\frac{1-\beta}{2}} \leq (t_2 - t_1)^{\frac{1-\beta}{2}}.
\] 
Therefore the two estimates (\ref{eq:holder_d1_estimate}) and (\ref{eq:holder_d2_estimate}) complete the proof.
\end{proof}
\noindent By again using similar methods to the proof of \Cref{corollary:u_linfty_bound} as well as using said corollary itself, we can now derive our desired parabolic Hölder bound for the third solution component $u_\eps$.
\begin{corollary}\label{corollary:time_hoelder}
	There exists $t_0 > 0$ such that, for each $\alpha \in (0,\frac{1}{5})$, there is a constant $C(\alpha) > 0$ with
	\[
	\|u_\eps\|_{C^\alpha([t,t+1]; C^{1+\alpha}(\overline{\Omega}))} \leq C(\alpha)
	\]
	for all $t > t_0$ and $\eps \in (0,1)$.
\end{corollary}
\begin{proof}
	Let $\beta \defs 1 - 2\alpha \in (\frac{3}{5},1)$ and $p \in (2,\infty)$ be such that
	\[
		1 + \alpha < 2\beta - \frac{2}{p},
	\] 
	which is always possible because $1+\alpha < \frac{6}{5} < 2\beta$.
	Then \Cref{corollary:u_linfty_bound}  and \Cref{lemma:time_hoelder_proto} give us a parameter independent time $t_0 > 0$ and constant $K_1 > 0$ such that
	\begin{equation}	
		\|u_\eps(\cdot, t_1) \|_{C^{1+\alpha}(\overline{\Omega})} \leq K_1 \label{eq:space_hoelder}
	\end{equation}
	and
	\[
		\|A^\beta \left[\, u_\eps(\cdot, t_2) - u_\eps(\cdot, t_1) \,\right] \|_\L{p} \leq K_1 (t_2 - t_1)^\alpha
	\]
	for all $\eps \in (0,1)$ and $t_2 > t_1 > t_0$ with $t_2 - t_1 < 1$, which by the continuous embedding of $D(A^\beta_p)$ into $C^{1+\alpha}(\overline{\Omega})$ (cf.\ \cite[p.~39]{DanHenryGeom} or \cite{GigaDomainsFractionalPowers1985}) implies that
	\begin{equation}
		\|u_\eps(\cdot, t_2) - u_\eps(\cdot, t_1) \|_{C^{1+\alpha}(\overline{\Omega})} \leq K_1 K_2 (t_2 - t_1)^\alpha \label{eq:time_hoelder}
	\end{equation}
	for all $\eps \in (0,1)$ with some $K_2 > 0$. Combining (\ref{eq:space_hoelder}) and (\ref{eq:time_hoelder}) then implies the desired result.
\end{proof}
\noindent
To now prove similar, albeit slightly weaker, parabolic Hölder bounds for the first two solution components of our approximate solutions, we will employ the ready-made parabolic regularity theory of \cite{PorzioVespriHoelder} to the first two equations in (\ref{approx_system}) to get a similar uniform Hölder bound.
\begin{lemma} \label{lemma:time_hoelder_2}
	There exists $t_0 > 0$, $\alpha \in (0,1)$ and $C > 0$ such that
	\[
		\|n_\eps\|_{C^{\alpha, \frac{\alpha}{2}}(\overline{\Omega}\times[t,t+1])} + \|c_\eps\|_{C^{\alpha, \frac{\alpha}{2}}(\overline{\Omega}\times[t,t+1])} \leq C
	\]
	for all $t > t_0$ and $\eps \in (0,1)$.
\end{lemma}
\begin{proof}
	As preparations for the proof, let us now fix $t_0 > 0$ and $K_1 > 0$ such that
	\begin{equation} \label{eq:bounds_needed_for_holder}
	\begin{aligned}
		\|n_\eps\|_{L^\infty(\Omega\times[t_0-1, \infty))} \leq K_1, &\;\;\;\; \|\grad c_\eps\|_{L^\infty(\Omega\times[t_0-1,\infty))} \leq K_1, \\  \|c_\eps\|_{L^\infty(\Omega\times[t_0-1,\infty))} \leq K_1, &\;\;\;\; \|u_\eps\|_{L^\infty(\Omega\times[t_0-1,\infty))} \leq K_1
	\end{aligned}
	\end{equation}
	due to \Cref{lemma:n_linfty_bound}, \Cref{lemma:grad_c_bound}, \Cref{lemma:basic_props} and \Cref{corollary:u_linfty_bound}.
	\\[0.5em]
	Let us now check that our approximate solutions conform to the prerequisites used for Theorem 1.3 in \cite{PorzioVespriHoelder}. Framed in the notation of the reference, the first two equations in (\ref{approx_system}) considered in isolation translate to 
	\begin{equation}
		\begin{aligned}
		&a(x,t,y,z) \defs z - n_\eps(x,t) S_\eps(x,n_\eps,c_\eps) \grad c_\eps(x, t) - u_\eps(x, t) n_\eps(x, t), &&\hspace{-1px}b(x,t,y,z) \defs  0,\\
		&a(x,t,y,z) \defs z \,\,\,\hphantom{ - n_\eps(x,t) S_\eps(x,n_\eps,c_\eps) \grad c_\eps(x, t)} - u_\eps(x, t) c_\eps(x, t), &&\hspace{-1px}b(x,t,y,z) \defs  -n_\eps(x,t) f(c_\eps(x,t)) \end{aligned}
		\label{eq:porzio_vespri_notation}
	\end{equation}
	for $(x,t,y,z) \in \Omega\times[t_0-1, \infty)\times\R\times\R^2$, respectively. We can then choose the parameters in the reference to be mostly constants, which only depend on $K_1$ due to (\ref{eq:bounds_needed_for_holder}) and (\ref{eq:porzio_vespri_notation}), or to be trivial. We further choose $p \defs  2$, $r\;\hat{} \defs \infty$, $q\;\hat{} \defs 2$, $\kappa_1 \defs \frac{1}{2}$. The remaining conditions to use Theorem 1.3 from \cite{PorzioVespriHoelder} are then easy to check and we can therefore apply it to the first and second equation in (\ref{approx_system}) to gain $\alpha \in (0,1), K_2 > 0$ such that
	\begin{equation*} 
		\|n_\eps\|_{C^{\alpha, \frac{\alpha}{2}}(\overline{\Omega}\times[t, t + 1])} + \|c_\eps\|_{C^{\alpha, \frac{\alpha}{2}}(\overline{\Omega}\times[t, t + 1])} \leq K_2
	\end{equation*}
	for all $t > t_0$ and $\eps \in (0,1)$. As Theorem 1.3 makes it explicit, how $K_2$ and $\alpha$ depend on the chosen parameters, it is easy to verify that both constants are in fact independent of $t$. This completes the proof.
\end{proof}

\subsection{Deriving $C^{2+\alpha, 1+\frac{\alpha}{2}}$-type parabolic Hölder regularity properties for the generalized solutions $(n,c,u)$}

\label{section:transition_to_gen}

We now transition from proving properties of the approximate solutions to the generalized solutions $(n,c,u)$ themselves. This is done mostly to mitigate problems stemming from the approximated sensitivity term $S_\eps$ when applying higher order parabolic regularity theory to the approximate solutions.
\\[0.5em]
Our first step then is to translate the uniform parabolic Hölder regularity properties of the approximate solutions to our generalized solution $(n,c,u)$ by using the well-known Arzelà--Ascoli theorem.
\begin{lemma}\label{lemma:time_hoelder_intermediate}
	There exist $t_0 > 0$, $\alpha \in (0,1)$ and $C > 0$ such that
	\begin{equation*}
	\|n\|_{C^{\alpha, \frac{\alpha}{2}}(\overline{\Omega}\times[t,t+1])} + \|c\|_{C^{\alpha, \frac{\alpha}{2}}(\overline{\Omega}\times[t,t+1])} + \|u\|_{C^\alpha([t,t+1];C^{1+\alpha}(\overline{\Omega}))} \leq C 
	\end{equation*}
	for all $t > t_0$.
\end{lemma}
\begin{proof}
	As we already know that $(n, c, u)$ are the (almost everywhere) pointwise limits of the approximate solutions $(n_\eps, c_\eps, u_\eps)$ along a sequence $(\eps_j)_{j\in\N}$, the compact embedding properties of the Hölder spaces $C^{\beta, \frac{\beta}{2}}(\overline{\Omega}\times[t,t+1])$, $C^\beta([t,t+1];C^{1+\beta}(\overline{\Omega}))$ into similar spaces with slightly smaller parameters combined with \Cref{corollary:time_hoelder} and \Cref{lemma:time_hoelder_2} directly yield our desired result.
\end{proof}\noindent
While the weak solution properties described in \Cref{definition:weak_solution} for the second and third equation are fairly standard and therefore very accessible to ready-made regularity theory, the integral inequalities used for the first solution component $n$ in said definition are not compatible with such standard regularity results. Therefore, our second step in this section is arguing that $n$ fulfills a similar weak solution property to the other two solution components from some time $t_0 > 0$ onward due to the strong a priori bounds derived in the previous sections.
\begin{lemma}\label{lemma:n_weak_solution}
	There exists $t_0 > 0$ such that $n$ is a weak solution of the first equation in (\ref{problem}) on $[t_0,\infty)$ with Neumann boundary conditions in the sense that $n \in C^0(\overline{\Omega}\times[t_0, \infty))\cap L^2_\loc([t_0, \infty);W^{1,2}(\Omega))$ and
	\begin{equation} \label{eq:n_weak_solution_definition}
	\int_{t_0}^\infty\int_\Omega n \varphi_t + \int_\Omega n(\cdot, t_0) \varphi(\cdot, t_0) = \int_{t_0}^\infty \int_\Omega \grad n \cdot \grad \varphi - \int_{t_0}^\infty \int_\Omega nS(\cdot,n,c) \grad c \cdot \grad \varphi - \int_{t_0}^\infty \int_\Omega n (u\cdot \grad \varphi)
	\end{equation}
	holds for all $\varphi \in C_0^\infty(\overline{\Omega}\times[t_0,\infty))$.
\end{lemma}
\begin{proof}
	We begin by fixing $t_0 > 0$ and $K > 0$ such that
	\begin{equation}\label{eq:weak_solution_bounds}
	\begin{aligned}
	&\|n_\eps\|_{L^\infty(\Omega\times[t_0, \infty))} \leq K, \;\;\;\; &&\|n_\eps\|_{L_\loc^2([t_0, \infty);W^{1,2}(\Omega))} \leq K, \\  
	&\|c_\eps\|_{L_\loc^2([t_0, \infty);W^{1,2}(\Omega))} \leq K,	&&\|u_\eps\|_{L^\infty(\Omega\times[t_0, \infty))} \leq K
	\end{aligned}
	\end{equation}
	due to  \Cref{lemma:n_linfty_bound}, \Cref{lemma:grad_n_integrability}, \Cref{lemma:basic_props} and \Cref{corollary:u_linfty_bound}.
	Among other things, this then allows us to assume without loss of generality (by potentially choosing a subsequence) that
	\begin{equation}\label{eq:n_weak_convergence}
	\grad n_\eps \rightharpoonup \grad n\stext{and}\grad c_\eps \rightharpoonup \grad c \;\;\;\; \text{ in } L_\loc^2(\Omega\times[t_0, \infty)) \text{ as } \eps = \eps_j \searrow 0.
	\end{equation}
	This combined with \Cref{lemma:time_hoelder_intermediate} then immediately gives us that $n \in C^0(\overline{\Omega}\times[t_0, \infty))\cap L^2_\loc([t_0, \infty);W^{1,2}(\Omega))$.
	\\[0.5em]
	It is further easy to see that the approximate solutions satisfy
	\begin{align*}
	&\int_{t_0}^\infty\int_\Omega n_\eps \varphi_t + \int_\Omega n_\eps(\cdot, t_0) \varphi(\cdot, t_0) 
	\\=& \int_{t_0}^\infty \int_\Omega \grad n_\eps \cdot \grad \varphi - \int_{t_0}^\infty \int_\Omega n_\eps S_\eps(\cdot,n_\eps,c_\eps) \grad c_\eps \cdot \grad \varphi - \int_{t_0}^\infty \int_\Omega n_\eps (u_\eps\cdot \grad \varphi) \numberthis \label{eq:n_approx_weak_solution_definition}
	\end{align*}
	for all $\varphi \in C_0^\infty(\overline{\Omega}\times[t_0,\infty))$. We therefore now only need to show that all of the above integral terms converge along the sequence $(\eps_j)_{j\in\N}$ to the corresponding terms in (\ref{eq:n_weak_solution_definition}). For the first two and the last term in (\ref{eq:n_approx_weak_solution_definition}), convergence is immediately assured by the dominated convergence theorem in combination with some of the bounds from (\ref{eq:weak_solution_bounds}). The convergence of the third term in (\ref{eq:n_approx_weak_solution_definition}) follows directly from (\ref{eq:n_weak_convergence}). For the remaining fourth term in (\ref{eq:n_approx_weak_solution_definition}), consider first that (\ref{eq:weak_solution_bounds}) combined with the (almost everywhere) pointwise convergence of the approximate solutions implies that $n_\eps S_\eps(x,n_\eps, c_\eps) \grad \varphi \rightarrow n S(x,n, c)\grad \varphi$ in $L^2(\Omega\times[t_0, \infty))$ as $\eps = \eps_j \searrow 0$ due to the dominated convergence theorem and the fact that the approximate sensitivities $S_\eps$ converge to $S$ in a pointwise fashion as $\eps \searrow 0$. This then combined with the weak convergence properties from (\ref{eq:n_weak_convergence}) ensures the convergence of the last remaining term and therefore completes the proof.
\end{proof}
\noindent
Given that now $n$, $c$, and $u$ each fulfill a quite standard weak solution property for their corresponding equations and are already of fairly high regularity (cf.\ \Cref{lemma:time_hoelder_intermediate}), the last step before the final proof of this paper is to use the well-known regularity theory from \cite{ladyvzenskaja1988linear} and \cite{LiebermanRegularity} (for the first two equations) as well as \cite{MR1928881} and \cite{MR2343611} (for the third equation) combined with a standard cut-off argument to remove the influence of initial data regularity to argue that all three solution components were already bounded in some $C^{2+\alpha, 1+\frac{\alpha}{2}}(\overline{\Omega}\times[t,t+1])$ spaces. As this argument is essentially identical to the one employed in \cite[Lemma 5.7]{MR3531759}, we will not unnecessarily reiterate the relevant arguments here, but refer the reader to the literature for the final step up in terms of regularity. 
\begin{lemma} \label{lemma:highest_hoelder_norms}
	There exist $t_0 > 0$, $\alpha \in (0,1)$ and $C > 0$ such that
	\[
	\|n\|_{C^{2+\alpha, 1+\frac{\alpha}{2}}(\overline{\Omega}\times[t,t+1])} + \|c\|_{C^{2+\alpha, 1+\frac{\alpha}{2}}(\overline{\Omega}\times[t,t+1])} + \|u\|_{C^{2+\alpha, 1+\frac{\alpha}{2}}(\overline{\Omega}\times[t,t+1])} \leq C
	\]
	for all $t > t_0$. 
\end{lemma}
\begin{proof}
	This can be directly seen via a straightforward adaptation of \cite[Lemma 5.7]{MR3531759}.
\end{proof}

\subsection{Proof of \Cref{theorem:eventual_smoothness}}
Having now essentially established all our desired regularity properties for the generalized solutions in the previous lemma, we now focus on proving the long time stabilization properties for $n, c$ and $u$ in $C^2(\overline{\Omega})$ outlined in \Cref{theorem:eventual_smoothness} as the penultimate proof of this paper.

\begin{lemma}\label{lemma:long_time_behavior} The generalized solution $(n,c,u)$ has the long time stabilization property (\ref{eq:eventual_convergence}).
\end{lemma}
\begin{proof}
	\Cref{lemma:highest_hoelder_norms} directly gives us $\alpha \in (0,1), t_0 > 0$ and $K_1 > 0$ such that
	\[
	\|n\|_{C^{2+\alpha, 1+\frac{\alpha}{2}}(\overline{\Omega}\times[t,t+1])} + \|c\|_{C^{2+\alpha, 1+\frac{\alpha}{2}}(\overline{\Omega}\times[t,t+1])} + \|u\|_{C^{2+\alpha, 1+\frac{\alpha}{2}}(\overline{\Omega}\times[t,t+1])} \leq K_1
	\]
	for all $t > t_0$.
	\\[0.5em]  
	Let us further fix $K_2 > 0$ such that
	\[
	\|\varphi - \overline{\varphi}\|_\L{1} \leq K_2 \sqrt{\int_\Omega \varphi} \cdot \sqrt{ \int_\Omega \varphi \ln\left(  \frac{ \; \varphi \; }{\overline{\varphi}} \right)} \;\;\;\; \text{ for all nonnegative } \varphi \in C^0(\overline{\Omega}) \text{ with } \overline{\varphi} \defs \frac{1}{|\Omega|} \int_\Omega \varphi
	\]
	according to a Czisz\'{a}r--Kullback or Pinsker-type inequality (cf.\ \cite{CsiszarKullbackInequality}). Let now $\delta > 0$. By \Cref{lemma:things_are_small} and the inequality above, there therefore exists $t_\delta > t_0$ such that
	\begin{equation} \label{eq:n_l1_uniform_convergence}
	\|n_\eps(\cdot, t) - \overline{n_0}\|_{L^1(\Omega)} \leq K_2 \sqrt{\int_\Omega n_0 } \cdot \sqrt{ \int_\Omega n_\eps(\cdot, t) \ln \left( \frac{\; n_\eps(\cdot, t) \;}{\overline{n_0}} \right)} < \frac{\delta}{2}
	\end{equation}
	for all $t > t_\delta$ and $\eps \in (0,1)$.
	Further for each $t > t_\delta > t_0$, there exists an $\eps(t) \in (0,1)$ such that
	\begin{equation*} 
	\|n(\cdot,t) - n_{\eps(t)}(\cdot, t)\|_\L{1} < \frac{\delta}{2}
	\end{equation*}
	because of, for example, the (almost everywhere) pointwise convergence of the approximate solutions to the generalized solutions combined with the dominated convergence theorem (using a constant majorant as established by \Cref{lemma:n_linfty_bound}). Combining the above two inequalities then results in
	\[
	\|n(\cdot, t) - \overline{n_0}\|_\L{1} \leq \|n(\cdot, t) - n_{\eps(t)}(\cdot, t)\|_\L{1} + \|n_{\eps(t)}(\cdot, t) - \overline{n_0}\|_\L{1} < \delta
	\]
	for all $t > t_\delta$ and therefore $n(\cdot, t) \rightarrow \overline{n_0}$  in $\L{1}$ as $t \rightarrow \infty$.
	\\[0.5em]
	As the start of a proof by contradiction, we assume now that $n(\cdot, t)$ does not converge to $\overline{n_0}$ in $C^2(\overline{\Omega})$ as $t \rightarrow \infty$. Then there must exist a constant $K_3 > 0$ and a sequence $(t_k)_{k\in\N}$ with $t_k \rightarrow \infty$ as $k \rightarrow \infty$ such that
	\begin{equation}\label{eq:subsequence_with_distance}
		\|n(\cdot, t_k) - \overline{n_0}\|_{C^2(\overline{\Omega})} > K_3 \;\;\;\; \text{ for all } k \in \N.
	\end{equation}
	As the family
	\[
		(n(\cdot, t_k))_{k \in \N}		
	\]
	is furthermore uniformly bounded in $C^{2+\alpha}(\overline{\Omega})$ by $K_1$, an application of the Arzelà--Ascoli theorem yields that the sequence $(t_k)_{k\in\N}$ has a subsequence, along which $n(\cdot, t_k)$ converges to some limit value in $C^2(\overline{\Omega})$. As we already know that $n(\cdot, t_k)$ converges to $\overline{n_0}$ in $L^1(\Omega)$ as $k\rightarrow \infty$ by prior arguments, the above $C^2(\overline{\Omega})$ limit must be $\overline{n_0}$ as well. This is a contraction to (\ref{eq:subsequence_with_distance}) and therefore we have proven that $n(\cdot, t) \rightarrow \overline{n_0}$ as $t \rightarrow \infty$ in $C^2(\overline{\Omega})$.
	\\[0.5em]
	As we have proven similar uniform convergence properties to (\ref{eq:n_l1_uniform_convergence}) for $c_\eps(\cdot, t)$ and $u_\eps(\cdot, t)$ in \Cref{lemma:c_stabilization} and \Cref{lemma:things_are_small}, the above argument can basically be reused verbatim to prove the remaining two convergence properties in (\ref{eq:eventual_convergence}). This completes the proof.
\end{proof}

\noindent 
This now allows us to tackle the last argument of this paper, namely the proof of \Cref{theorem:eventual_smoothness}.

\begin{proof}[Proof of \Cref{theorem:eventual_smoothness}]
	A combination of \Cref{lemma:highest_hoelder_norms} and \Cref{lemma:long_time_behavior} now grants us all desired properties for $(n,c,u)$ from some time $t_0 > 0$ onward as it is well-known that weak solutions of the kind characterized in \Cref{definition:weak_solution} and \Cref{lemma:n_weak_solution} with regularity properties as provided by \Cref{lemma:highest_hoelder_norms} are in fact classical already and an associated pressure function $P$ for the fluid equation can be constructed (cf.\ \cite{ladyvzenskaja1988linear}, \cite{MR1928881}).
\end{proof}

\section*{Acknowledgment} The author acknowledges support of the \emph{Deutsche Forschungsgemeinschaft} in the context of the project \emph{Fine structures in interpolation inequalities and application to parabolic problems}, project number 462888149.


\begin{thebibliography}{10}

	\bibitem{MoserTrudingerOpt1}
	\textsc{Adachi, S.} and \textsc{Tanaka, K.}:
	\newblock {\em Trudinger type inequalities in {$\bold R^N$} and their best
	  exponents}.
	\newblock Proc. Amer. Math. Soc., 128(7):2051--2057, 2000.
	\newblock \href {http://dx.doi.org/10.1090/S0002-9939-99-05180-1}
	  {\path{doi:10.1090/S0002-9939-99-05180-1}}.
	
	\bibitem{AdimurthiGlobalCompactnessProperties2000}
	\textsc{{Adimurthi}} and \textsc{Struwe, M.}:
	\newblock {\em Global Compactness Properties of Semilinear Elliptic Equations
	  with Critical Exponential Growth}.
	\newblock Journal of Functional Analysis, 175(1):125--167, 2000.
	\newblock \href {http://dx.doi.org/10.1006/jfan.2000.3602}
	  {\path{doi:10.1006/jfan.2000.3602}}.
	
	\bibitem{AlvesExistenceUniformDecay2009a}
	\textsc{Alves, C.~O.} and \textsc{Cavalcanti, M.~M.}:
	\newblock {\em On Existence, Uniform Decay Rates and Blow up for Solutions of
	  the 2-{{D}} Wave Equation with Exponential Source}.
	\newblock Calculus of Variations and Partial Differential Equations,
	  34(3):377--411, 2009.
	\newblock \href {http://dx.doi.org/10.1007/s00526-008-0188-z}
	  {\path{doi:10.1007/s00526-008-0188-z}}.
	
	\bibitem{ArrietaParabolicProblemsNonlinear1999}
	\textsc{Arrieta, J.~M.}, \textsc{Carvalho, A.~N.}, and
	  \textsc{{Rodr{\'i}guez-Bernal}, A.}:
	\newblock {\em Parabolic Problems with Nonlinear Boundary Conditions and
	  Critical Nonlinearities}.
	\newblock Journal of Differential Equations, 156(2):376--406, 1999.
	\newblock \href {http://dx.doi.org/10.1006/jdeq.1998.3612}
	  {\path{doi:10.1006/jdeq.1998.3612}}.
	
	\bibitem{MR3351175}
	\textsc{Bellomo, N.}, \textsc{Bellouquid, A.}, \textsc{Tao, Y.}, and
	  \textsc{Winkler, M.}:
	\newblock {\em Toward a mathematical theory of {K}eller--{S}egel models of
	  pattern formation in biological tissues}.
	\newblock Math. Models Methods Appl. Sci., 25(9):1663--1763, 2015.
	\newblock \href {http://dx.doi.org/10.1142/S021820251550044X}
	  {\path{doi:10.1142/S021820251550044X}}.
	
	\bibitem{BrezisFAandPDE}
	\textsc{Brezis, H.}:
	\newblock {\em Functional analysis, {S}obolev spaces and partial differential
	  equations}.
	\newblock Universitext. Springer, New York, 2011.
	
	\bibitem{MR3562314}
	\textsc{Cao, X.}:
	\newblock {\em Global classical solutions in chemotaxis(--{N}avier)--{S}tokes
	  system with rotational flux term}.
	\newblock J. Differential Equations, 261(12):6883--6914, 2016.
	\newblock \href {http://dx.doi.org/10.1016/j.jde.2016.09.007}
	  {\path{doi:10.1016/j.jde.2016.09.007}}.
	
	\bibitem{MR3531759}
	\textsc{Cao, X.} and \textsc{Lankeit, J.}:
	\newblock {\em Global classical small-data solutions for a three-dimensional
	  chemotaxis {N}avier--{S}tokes system involving matrix-valued sensitivities}.
	\newblock Calc. Var. Partial Differential Equations, 55(4):Art. 107, 39, 2016.
	\newblock \href {http://dx.doi.org/10.1007/s00526-016-1027-2}
	  {\path{doi:10.1007/s00526-016-1027-2}}.
	
	\bibitem{MR3208807}
	\textsc{Chae, M.}, \textsc{Kang, K.}, and \textsc{Lee, J.}:
	\newblock {\em Global existence and temporal decay in {K}eller--{S}egel models
	  coupled to fluid equations}.
	\newblock Comm. Partial Differential Equations, 39(7):1205--1235, 2014.
	\newblock \href {http://dx.doi.org/10.1080/03605302.2013.852224}
	  {\path{doi:10.1080/03605302.2013.852224}}.
	
	\bibitem{chang1988conformal}
	\textsc{Chang, S.-Y.~A.} and \textsc{Yang, P.~C.}:
	\newblock {\em Conformal deformation of metrics on {$S^2$}}.
	\newblock J. Differential Geom., 27(2):259--296, 1988.
	\newblock URL: \url{http://projecteuclid.org/euclid.jdg/1214441783}.
	
	\bibitem{MoserTrudingerOpt2}
	\textsc{Cohn, W.~S.} and \textsc{Lu, G.~Z.}:
	\newblock {\em Best constants for {M}oser--{T}rudinger inequalities,
	  fundamental solutions and one-parameter representation formulas on groups of
	  {H}eisenberg type}.
	\newblock Acta Math. Sin. (Engl. Ser.), 18(2):375--390, 2002.
	\newblock \href {http://dx.doi.org/10.1007/s101140200159}
	  {\path{doi:10.1007/s101140200159}}.
	
	\bibitem{CsiszarKullbackInequality}
	\textsc{Csisz\'{a}r, I.}:
	\newblock {\em Information-type measures of difference of probability
	  distributions and indirect observations}.
	\newblock Studia Sci. Math. Hungar., 2:299--318, 1967.
	
	\bibitem{EllipticFunctionSpaces}
	\textsc{Demengel, F.} and \textsc{Demengel, G.}:
	\newblock {\em Functional spaces for the theory of elliptic partial
	  differential equations}.
	\newblock Universitext. Springer, London; EDP Sciences, Les Ulis, 2012.
	\newblock Translated from the 2007 French original by Reinie Ern\'{e}.
	\newblock \href {http://dx.doi.org/10.1007/978-1-4471-2807-6}
	  {\path{doi:10.1007/978-1-4471-2807-6}}.
	
	\bibitem{MR3377875}
	\textsc{Dolbeault, J.}, \textsc{Esteban, M.~J.}, and \textsc{Jankowiak, G.}:
	\newblock {\em The {M}oser--{T}rudinger--{O}nofri inequality}.
	\newblock Chin. Ann. Math. Ser. B, 36(5):777--802, 2015.
	\newblock \href {http://dx.doi.org/10.1007/s11401-015-0976-7}
	  {\path{doi:10.1007/s11401-015-0976-7}}.
	
	\bibitem{PhysRevLett.93.098103}
	\textsc{Dombrowski, C.}, \textsc{Cisneros, L.}, \textsc{Chatkaew, S.},
	  \textsc{Goldstein, R.~E.}, and \textsc{Kessler, J.~O.}:
	\newblock {\em Self-concentration and large-scale coherence in bacterial
	  dynamics}.
	\newblock Phys. Rev. Lett., 93:098103, 2004.
	\newblock \href {http://dx.doi.org/10.1103/PhysRevLett.93.098103}
	  {\path{doi:10.1103/PhysRevLett.93.098103}}.
	
	\bibitem{MR2754058}
	\textsc{Duan, R.}, \textsc{Lorz, A.}, and \textsc{Markowich, P.}:
	\newblock {\em Global solutions to the coupled chemotaxis-fluid equations}.
	\newblock Comm. Partial Differential Equations, 35(9):1635--1673, 2010.
	\newblock \href {http://dx.doi.org/10.1080/03605302.2010.497199}
	  {\path{doi:10.1080/03605302.2010.497199}}.
	
	\bibitem{EspositoExistenceBlowingupSolutions2005}
	\textsc{Esposito, P.}, \textsc{Grossi, M.}, and \textsc{Pistoia, A.}:
	\newblock {\em On the Existence of Blowing-up Solutions for a Mean Field
	  Equation}.
	\newblock Ann. Inst. H. Poincar\'e C Anal. Non Lin\'eaire, 22(2):227--257,
	  2005.
	\newblock \href {http://dx.doi.org/10.1016/j.anihpc.2004.12.001}
	  {\path{doi:10.1016/j.anihpc.2004.12.001}}.
	
	\bibitem{FriedmanPDE}
	\textsc{Friedman, A.}:
	\newblock {\em Partial differential equations}.
	\newblock Holt, Rinehart and Winston, Inc., New York-Montreal, Que.-London,
	  1969.
	
	\bibitem{GaldiNavierStokes1}
	\textsc{Galdi, G.~P.}:
	\newblock {\em An introduction to the mathematical theory of the
	  {N}avier--{S}tokes equations}.
	\newblock Springer Monographs in Mathematics. Springer, New York, second
	  edition, 2011.
	\newblock Steady-state problems.
	\newblock \href {http://dx.doi.org/10.1007/978-0-387-09620-9}
	  {\path{doi:10.1007/978-0-387-09620-9}}.
	
	\bibitem{MR3052352}
	\textsc{Ghoussoub, N.} and \textsc{Moradifam, A.}:
	\newblock {\em Functional inequalities: New perspectives and new applications},
	  volume 187 of {\em Mathematical Surveys and Monographs}.
	\newblock American Mathematical Society, Providence, RI, 2013.
	\newblock \href {http://dx.doi.org/10.1090/surv/187}
	  {\path{doi:10.1090/surv/187}}.
	
	\bibitem{MR605289}
	\textsc{Giga, Y.}:
	\newblock {\em The {S}tokes operator in {$L_{r}$} spaces}.
	\newblock Proc. Japan Acad. Ser. A Math. Sci., 57(2):85--89, 1981.
	\newblock URL: \url{http://projecteuclid.org/euclid.pja/1195516533}.
	
	\bibitem{GigaDomainsFractionalPowers1985}
	\textsc{Giga, Y.}:
	\newblock {\em Domains of Fractional Powers of the {{Stokes}} Operator in
	  {{L}}{\textsubscript{r}} Spaces}.
	\newblock Archive for Rational Mechanics and Analysis, 89(3):251--265, 1985.
	\newblock \href {http://dx.doi.org/10.1007/BF00276874}
	  {\path{doi:10.1007/BF00276874}}.
	
	\bibitem{GigaSolutionsSemilinearParabolic1986}
	\textsc{Giga, Y.}:
	\newblock {\em Solutions for Semilinear Parabolic Equations in
	  {{L\textsuperscript{p}}} and Regularity of Weak Solutions of the
	  {{Navier-Stokes}} System}.
	\newblock Journal of Differential Equations, 62(2):186--212, 1986.
	\newblock \href {http://dx.doi.org/10.1016/0022-0396(86)90096-3}
	  {\path{doi:10.1016/0022-0396(86)90096-3}}.
	
	\bibitem{MyExistence}
	\textsc{Heihoff, F.}:
	\newblock {\em Global Mass-Preserving Solutions for a Two-Dimensional
	  Chemotaxis System with Rotational Flux Components Coupled with a Full
	  {{Navier}}\textendash{{Stokes}} Equation}.
	\newblock Discrete \& Continuous Dynamical Systems - B, 22(11):0--0, 2020.
	\newblock \href {http://dx.doi.org/10.3934/dcdsb.2020120}
	  {\path{doi:10.3934/dcdsb.2020120}}.
	
	\bibitem{DanHenryGeom}
	\textsc{Henry, D.}:
	\newblock {\em Geometric theory of semilinear parabolic equations}, volume 840
	  of {\em Lecture Notes in Mathematics}.
	\newblock Springer-Verlag, Berlin-New York, 1981.
	
	\bibitem{HorstmannBlowupChemotaxisModel2001}
	\textsc{Horstmann, D.} and \textsc{Wang, G.}:
	\newblock {\em Blow-up in a Chemotaxis Model without Symmetry Assumptions}.
	\newblock European Journal of Applied Mathematics, 12(2):159--177, 2001.
	\newblock \href {http://dx.doi.org/10.1017/S0956792501004363}
	  {\path{doi:10.1017/S0956792501004363}}.
	
	\bibitem{keller1970initiation}
	\textsc{Keller, E.~F.} and \textsc{Segel, L.~A.}:
	\newblock {\em Initiation of slime mold aggregation viewed as an instability}.
	\newblock J. Theoret. Biol., 26(3):399--415, 1970.
	\newblock \href {http://dx.doi.org/10.1016/0022-5193(70)90092-5}
	  {\path{doi:10.1016/0022-5193(70)90092-5}}.
	
	\bibitem{ladyvzenskaja1988linear}
	\textsc{Lady{\v{z}}enskaja, O.~A.}, \textsc{Solonnikov, V.~A.}, and
	  \textsc{Ural'ceva, N.~N.}:
	\newblock {\em Linear and quasi-linear equations of parabolic type}, volume~23.
	\newblock American Mathematical Soc., 1988.
	
	\bibitem{MR3302296}
	\textsc{Li, T.}, \textsc{Suen, A.}, \textsc{Winkler, M.}, and \textsc{Xue, C.}:
	\newblock {\em Global small-data solutions of a two-dimensional chemotaxis
	  system with rotational flux terms}.
	\newblock Math. Models Methods Appl. Sci., 25(4):721--746, 2015.
	\newblock \href {http://dx.doi.org/10.1142/S0218202515500177}
	  {\path{doi:10.1142/S0218202515500177}}.
	
	\bibitem{LiebermanRegularity}
	\textsc{Lieberman, G.~M.}:
	\newblock {\em H\"{o}lder continuity of the gradient of solutions of uniformly
	  parabolic equations with conormal boundary conditions}.
	\newblock Ann. Mat. Pura Appl. (4), 148:77--99, 1987.
	\newblock \href {http://dx.doi.org/10.1007/BF01774284}
	  {\path{doi:10.1007/BF01774284}}.
	
	\bibitem{MR601389}
	\textsc{Mikha\u{\i}lov, V.~P.}:
	\newblock {\em Partial differential equations}.
	\newblock ``Mir'', Moscow; distributed by Imported Publications, Inc., Chicago,
	  Ill., 1978.
	\newblock Translated from the Russian by P. C. Sinha.
	
	\bibitem{moser1971sharp}
	\textsc{Moser, J.}:
	\newblock {\em A sharp form of an inequality by {N}. {T}rudinger}.
	\newblock Indiana Univ. Math. J., 20:1077--1092, 1970/71.
	\newblock \href {http://dx.doi.org/10.1512/iumj.1971.20.20101}
	  {\path{doi:10.1512/iumj.1971.20.20101}}.
	
	\bibitem{nagai1997application}
	\textsc{Nagai, T.}, \textsc{Senba, T.}, and \textsc{Yoshida, K.}:
	\newblock {\em Application of the {T}rudinger-{M}oser inequality to a parabolic
	  system of chemotaxis}.
	\newblock Funkcial. Ekvac., 40(3):411--433, 1997.
	\newblock URL: \url{http://www.math.kobe-u.ac.jp/~fe/xml/mr1610709.xml}.
	
	\bibitem{GNIClassic}
	\textsc{Nirenberg, L.}:
	\newblock {\em On elliptic partial differential equations}.
	\newblock Ann. Scuola Norm. Sup. Pisa Cl. Sci. (3), 13:115--162, 1959.
	
	\bibitem{MR677001}
	\textsc{Onofri, E.}:
	\newblock {\em On the positivity of the effective action in a theory of random
	  surfaces}.
	\newblock Comm. Math. Phys., 86(3):321--326, 1982.
	\newblock URL: \url{http://projecteuclid.org/euclid.cmp/1103921772}.
	
	\bibitem{PorzioVespriHoelder}
	\textsc{Porzio, M.~M.} and \textsc{Vespri, V.}:
	\newblock {\em H\"{o}lder estimates for local solutions of some doubly
	  nonlinear degenerate parabolic equations}.
	\newblock J. Differential Equations, 103(1):146--178, 1993.
	\newblock \href {http://dx.doi.org/10.1006/jdeq.1993.1045}
	  {\path{doi:10.1006/jdeq.1993.1045}}.
	
	\bibitem{MR1928881}
	\textsc{Sohr, H.}:
	\newblock {\em The {N}avier--{S}tokes equations}.
	\newblock Birkh\"{a}user Advanced Texts: Basler Lehrb\"{u}cher. [Birkh\"{a}user
	  Advanced Texts: Basel Textbooks]. Birkh\"{a}user Verlag, Basel, 2001.
	\newblock An elementary functional analytic approach.
	\newblock \href {http://dx.doi.org/10.1007/978-3-0348-8255-2}
	  {\path{doi:10.1007/978-3-0348-8255-2}}.
	
	\bibitem{MR2343611}
	\textsc{Solonnikov, V.~A.}:
	\newblock {\em Schauder estimates for the evolutionary generalized {S}tokes
	  problem}.
	\newblock In {\em Nonlinear equations and spectral theory}, volume 220 of {\em
	  Amer. Math. Soc. Transl. Ser. 2}, pages 165--200. Amer. Math. Soc.,
	  Providence, RI, 2007.
	\newblock \href {http://dx.doi.org/10.1090/trans2/220/08}
	  {\path{doi:10.1090/trans2/220/08}}.
	
	\bibitem{trudinger1967imbeddings}
	\textsc{Trudinger, N.~S.}:
	\newblock {\em On imbeddings into {O}rlicz spaces and some applications}.
	\newblock J. Math. Mech., 17:473--483, 1967.
	\newblock \href {http://dx.doi.org/10.1512/iumj.1968.17.17028}
	  {\path{doi:10.1512/iumj.1968.17.17028}}.
	
	\bibitem{Tuval2277}
	\textsc{Tuval, I.}, \textsc{Cisneros, L.}, \textsc{Dombrowski, C.},
	  \textsc{Wolgemuth, C.~W.}, \textsc{Kessler, J.~O.}, and \textsc{Goldstein,
	  R.~E.}:
	\newblock {\em Bacterial swimming and oxygen transport near contact lines}.
	\newblock Proceedings of the National Academy of Sciences, 102(7):2277--2282,
	  2005.
	\newblock \href {http://dx.doi.org/10.1073/pnas.0406724102}
	  {\path{doi:10.1073/pnas.0406724102}}.
	
	\bibitem{MR3801284}
	\textsc{Wang, Y.}, \textsc{Winkler, M.}, and \textsc{Xiang, Z.}:
	\newblock {\em Global classical solutions in a two-dimensional
	  chemotaxis-{N}avier--{S}tokes system with subcritical sensitivity}.
	\newblock Ann. Sc. Norm. Super. Pisa Cl. Sci. (5), 18(2):421--466, 2018.
	\newblock \href {http://dx.doi.org/10.2422/2036-2145.201603_004}
	  {\path{doi:10.2422/2036-2145.201603_004}}.
	
	\bibitem{MR3401606}
	\textsc{Wang, Y.} and \textsc{Xiang, Z.}:
	\newblock {\em Global existence and boundedness in a
	  {K}eller--{S}egel--{S}tokes system involving a tensor-valued sensitivity with
	  saturation}.
	\newblock J. Differential Equations, 259(12):7578--7609, 2015.
	\newblock \href {http://dx.doi.org/10.1016/j.jde.2015.08.027}
	  {\path{doi:10.1016/j.jde.2015.08.027}}.
	
	\bibitem{MR3542964}
	\textsc{Wang, Y.} and \textsc{Xiang, Z.}:
	\newblock {\em Global existence and boundedness in a {K}eller--{S}egel-{S}tokes
	  system involving a tensor-valued sensitivity with saturation: the 3{D} case}.
	\newblock J. Differential Equations, 261(9):4944--4973, 2016.
	\newblock \href {http://dx.doi.org/10.1016/j.jde.2016.07.010}
	  {\path{doi:10.1016/j.jde.2016.07.010}}.
	
	\bibitem{WinklerSemigroupRegularity}
	\textsc{Winkler, M.}:
	\newblock {\em Aggregation vs. global diffusive behavior in the
	  higher-dimensional {K}eller--{S}egel model}.
	\newblock J. Differential Equations, 248(12):2889--2905, 2010.
	\newblock \href {http://dx.doi.org/10.1016/j.jde.2010.02.008}
	  {\path{doi:10.1016/j.jde.2010.02.008}}.
	
	\bibitem{WinklerExistence}
	\textsc{Winkler, M.}:
	\newblock {\em Global large-data solutions in a chemotaxis-({N}avier-){S}tokes
	  system modeling cellular swimming in fluid drops}.
	\newblock Comm. Partial Differential Equations, 37(2):319--351, 2012.
	\newblock \href {http://dx.doi.org/10.1080/03605302.2011.591865}
	  {\path{doi:10.1080/03605302.2011.591865}}.
	
	\bibitem{MR3149063}
	\textsc{Winkler, M.}:
	\newblock {\em Stabilization in a two-dimensional chemotaxis-{N}avier--{S}tokes
	  system}.
	\newblock Arch. Ration. Mech. Anal., 211(2):455--487, 2014.
	\newblock \href {http://dx.doi.org/10.1007/s00205-013-0678-9}
	  {\path{doi:10.1007/s00205-013-0678-9}}.
	
	\bibitem{MR3426095}
	\textsc{Winkler, M.}:
	\newblock {\em Boundedness and large time behavior in a three-dimensional
	  chemotaxis-{S}tokes system with nonlinear diffusion and general sensitivity}.
	\newblock Calc. Var. Partial Differential Equations, 54(4):3789--3828, 2015.
	\newblock \href {http://dx.doi.org/10.1007/s00526-015-0922-2}
	  {\path{doi:10.1007/s00526-015-0922-2}}.
	
	\bibitem{WinklerLargeDataGeneralized}
	\textsc{Winkler, M.}:
	\newblock {\em Large-data global generalized solutions in a chemotaxis system
	  with tensor-valued sensitivities}.
	\newblock SIAM J. Math. Anal., 47(4):3092--3115, 2015.
	\newblock \href {http://dx.doi.org/10.1137/140979708}
	  {\path{doi:10.1137/140979708}}.
	
	\bibitem{MR3542616}
	\textsc{Winkler, M.}:
	\newblock {\em Global weak solutions in a three-dimensional
	  chemotaxis-{N}avier--{S}tokes system}.
	\newblock Ann. Inst. H. Poincar\'{e} Anal. Non Lin\'{e}aire, 33(5):1329--1352,
	  2016.
	\newblock \href {http://dx.doi.org/10.1016/j.anihpc.2015.05.002}
	  {\path{doi:10.1016/j.anihpc.2015.05.002}}.
	
	\bibitem{MR3605965}
	\textsc{Winkler, M.}:
	\newblock {\em How far do chemotaxis-driven forces influence regularity in the
	  {N}avier--{S}tokes system?}
	\newblock Trans. Amer. Math. Soc., 369(5):3067--3125, 2017.
	\newblock \href {http://dx.doi.org/10.1090/tran/6733}
	  {\path{doi:10.1090/tran/6733}}.
	
	\bibitem{WinklerStokesCase}
	\textsc{Winkler, M.}:
	\newblock {\em Global mass-preserving solutions in a two-dimensional
	  chemotaxis-{S}tokes system with rotational flux components}.
	\newblock J. Evol. Equ., 18(3):1267--1289, 2018.
	\newblock \href {http://dx.doi.org/10.1007/s00028-018-0440-8}
	  {\path{doi:10.1007/s00028-018-0440-8}}.
	
	\bibitem{WinklerEventualSmoothness}
	\textsc{Winkler, M.}:
	\newblock {\em Can rotational fluxes impede the tendency toward spatial
	  homogeneity in nutrient taxis(-Stokes) systems?}
	\newblock International Mathematics Research Notices, 2019.
	\newblock \href {http://dx.doi.org/10.1093/imrn/rnz056}
	  {\path{doi:10.1093/imrn/rnz056}}.
	
	\bibitem{WinklerSmallMass}
	\textsc{Winkler, M.}:
	\newblock {\em Small-Mass Solutions in the Two-Dimensional {{Keller-Segel}}
	  System Coupled to the {{Navier-Stokes}} Equations}.
	\newblock Siam Journal On Mathematical Analysis, 52(2):2041--2080, 2020.
	\newblock \href {http://dx.doi.org/10.1137/19M1264199}
	  {\path{doi:10.1137/19M1264199}}.
	
	\bibitem{Xiong2018}
	\textsc{Xiong, J.}:
	\newblock {\em A derivation of the sharp Moser--Trudinger--Onofri inequalities
	  from the fractional Sobolev inequalities}.
	\newblock Peking Mathematical Journal, 1(2):221--229, 2018.
	\newblock \href {http://dx.doi.org/10.1007/s42543-019-00012-3}
	  {\path{doi:10.1007/s42543-019-00012-3}}.
	
	\bibitem{MR3294964}
	\textsc{Xue, C.}:
	\newblock {\em Macroscopic equations for bacterial chemotaxis: Integration of
	  detailed biochemistry of cell signaling}.
	\newblock J. Math. Biol., 70(1-2):1--44, 2015.
	\newblock \href {http://dx.doi.org/10.1007/s00285-013-0748-5}
	  {\path{doi:10.1007/s00285-013-0748-5}}.
	
	\bibitem{MR2505083}
	\textsc{Xue, C.} and \textsc{Othmer, H.~G.}:
	\newblock {\em Multiscale models of taxis-driven patterning in bacterial
	  populations}.
	\newblock SIAM J. Appl. Math., 70(1):133--167, 2009.
	\newblock \href {http://dx.doi.org/10.1137/070711505}
	  {\path{doi:10.1137/070711505}}.
	
	\bibitem{MR3369260}
	\textsc{Zhang, Q.} and \textsc{Li, Y.}:
	\newblock {\em Global weak solutions for the three-dimensional
	  chemotaxis--{N}avier--{S}tokes system with nonlinear diffusion}.
	\newblock J. Differential Equations, 259(8):3730--3754, 2015.
	\newblock \href {http://dx.doi.org/10.1016/j.jde.2015.05.012}
	  {\path{doi:10.1016/j.jde.2015.05.012}}.
	
	\end{thebibliography}
\end{document}